\documentclass[11pt, twoside, letterpaper]{amsart}
  
\usepackage[active]{srcltx}
\usepackage{t1enc}
\usepackage{latexsym}
\usepackage{amssymb}
\usepackage{graphicx}
\usepackage{amsmath}
\usepackage{amsthm}
\usepackage{amsfonts}
\usepackage{mathrsfs}
\usepackage[all]{xy}
\usepackage[british]{babel}
\usepackage{color}
\usepackage[hypertexnames=false,
    pdftex,
	pdfpagemode=UseNone,
	breaklinks=true,
	extension=pdf,
	colorlinks=true,
	linkcolor=blue,
	citecolor=blue,
	urlcolor=blue,
]{hyperref}

\newcommand*{\definiere}{\mathrel{\mathop:}=}

\newcommand{\cyclic}{\mathop{\kern0.9ex{{+}\kern-2.10ex\raise-0.20
      ex\hbox{\Large\hbox{$\circlearrowright$}}}}\limits}
\newcommand{\acts}{\mbox{ \raisebox{0.26ex}{\tiny{$\bullet$}} }}

\newtheoremstyle{daniel}{3.0mm}{2mm}{\itshape}{}{\bfseries}{.}{1.5mm}{}
\theoremstyle{daniel}
\newtheorem{thm}{Theorem}[section]
\newtheorem{prop}[thm]{Proposition}
\newtheorem{Defi}[thm]{Definition}
\newtheorem{lemma}[thm]{Lemma}
\newtheorem{cor}[thm]{Corollary}
\newtheorem{Exs}[thm]{Examples}
\newtheorem{Rems}[thm]{Remarks}

\newtheorem*{thm*}{Theorem}
\newtheorem*{cor*}{Corollary}
\newtheorem*{thm3.6}{Theorem 3.6}
\newtheorem*{thm4.3}{Theorem 4.3}
\newtheorem*{mainthm}{Main Theorem}
\newtheorem*{prop*}{Proposition}
\newtheorem*{Notation}{Notation}

\newtheorem{propnot}[thm]{Proposition and Notation}

\newtheorem{Lem}[thm]{Lemma}
\newtheorem{Cor}[thm]{Corollary}
\newtheorem{Prop}[thm]{Proposition}

\newtheorem{Def}[thm]{Definition}
\newtheorem{Not}[thm]{Notation}

\newtheorem{Rem}[thm]{Remark}

\newtheorem{Ex}[thm]{Example}

\newtheorem*{Setup}{Setup}

\newenvironment{rem}   {\begin{Rem}\em}{\end{Rem}}
\newenvironment{rems}   {\begin{Rems}\em}{\end{Rems}}
\newenvironment{defi}  {\begin{Defi}\em}{\end{Defi}}

\def\cE{\mathcal E}

\def\cO{\mathcal O}

\newcommand{\C}{\mathbb{C}}
\newcommand{\N}{\mathbb{N}}

\newcommand{\Q}{\mathbb{Q}}
\newcommand{\R}{\mathbb{R}}

\def\Amp{\mbox{Amp}}
\def\Chow{\mbox{Chow}}

\def\Hom{\mbox{Hom}}

\def\Quot{\mbox{Quot}}

\def\Sing{\mbox{Sing}}
\def\Supp{\mbox{Supp}}

\def\SL{S\negthinspace L}
\def\GL{G\negthinspace L}
\def\PGL{P G\negthinspace L}

\def\hq{\hspace{-0.5mm}/\hspace{-0.14cm}/ \hspace{-0.5mm}}

\DeclareMathOperator{\rk}{rk}

\numberwithin{equation}{section}

\DeclareRobustCommand{\SkipTocEntry}[5]{} 

\makeatletter
\def\l@subsection{\@tocline{2}{0pt}{2.5pc}{5pc}{}} 
\makeatother

\begin{document}
\title[Compact moduli spaces for slope-semistable sheaves]{Compact moduli spaces for \\slope-semistable sheaves}
\author{Daniel Greb}
\address{Daniel Greb\\Essener Seminar f\"ur Algebraische Geometrie und Arithmetik\\Fakult\"at f\"ur Ma\-the\-matik\\Universit\"at Duisburg--Essen\\
45117 Essen\\ Germany}
\email{daniel.greb@uni-due.de}
\urladdr{\href{http://www.esaga.uni-due.de/daniel.greb/}{http://www.esaga.uni-due.de/daniel.greb/}}

\author{Matei Toma}
\address{Matei Toma\\Institut de Math\'ematiques Elie Cartan,
Universit\'e de Lorraine\\
B.P. 239\\
54506 Vandoeuvre-l\`es-Nancy Cedex\\
France}
\email{Matei.Toma@univ-lorraine.fr}
\urladdr{\href{http://www.iecn.u-nancy.fr/~toma/}{http://www.iecn.u-nancy.fr/~toma/}}

\date{\today}


\begin{abstract}
We resolve pathological wall-crossing phenomena for mo\-du\-li spaces of sheaves on higher-dimensional complex projective manifolds. This is achieved by considering slope-semi\-stab\-ility with respect to movable curves rather than divisors. Moreover, given a projective $n$-fold and a curve $C$ that arises as the complete intersection of $(n-1)$ very ample divisors, we construct a modular compactification of the moduli space of vector bundles that are slope-stable with respect to $C$. Our construction generalises the algebro-geometric construction of the Do\-nald\-son-Uhlenbeck compactification by Joseph Le Potier and Jun Li. Furthermore, we describe the geometry of the newly construced moduli spaces by relating them to moduli spaces of simple sheaves and to Gieseker-Maruyama moduli spaces.
\end{abstract}
\maketitle
\vspace{2mm}

\setcounter{tocdepth}{2}
\tableofcontents

\vspace{-1cm}

\section{Introduction}
Moduli spaces of sheaves play a central role in Algebraic Geometry: they provide intensively studied examples of higher-dimensional varieties (e.g.~of hyperk\"ahler manifolds), they are naturally associated with the underlying variety and can therefore be used to define fine invariants of its differentiable structure, and they have found application in numerous problems of mathematical physics.

To obtain moduli spaces that exist as schemes rather than just stacks, it is necessary to choose a semistability condition that selects the objects for which a moduli spaces is to be constructed. In dimension greater than one, both Gieseker-semistability (which yields projective moduli spaces in
arbitrary dimension), and slope-semistability (which is better behaved geometrically, e.g.~with respect to
tensor products and restrictions) depend on a parameter, classically the class of a line bundle in the ample
cone of the underlying variety. As a consequence, with respect to all the points of view suggested above it is
of great importance to understand how the moduli space of semistable sheaves changes when the semistability parameter varies.

In the case where the underlying variety is of dimension two this problem has been investigated by a number
of authors and a rather complete geometric picture has emerged, which can be summarised as follows:

(i) There exists a projective moduli space for slope-semistable sheaves that compactifies the moduli space of slope-stable vector bundles. This moduli space is homeomorphic to the Donaldson-Uhlenbeck compactification, endowing the latter with a complex structure,
and admits a natural morphism from the Gieseker compactification. Motivated by Donaldson's non-vanishing result \cite[Thm.~C]{DonaldsonPolynomialInvariants}, this was proved independently by Le Potier \cite{LePotierDonaldsonUhlenbeck}, 
Li \cite{JunLiDonaldsonUhlenbeck}, and Morgan~\cite{MorganDonaldsonUhlenbeck}.

(ii) In the ample cone of the underlying variety there exists a locally finite chamber structure given by
linear rational walls, so that the notion of slope/Gieseker-semistability (and hence the moduli space)
does not change within the chambers, see \cite{Qin93}.

(iii) Moreover, at least when the second Chern class of the sheaves under consideration is sufficiently
big, moduli spaces corresponding to two chambers separated by a common wall are birational, and
the change in geometry can be understood by studying the moduli space of sheaves that are slope-semistable
with respect to the class of an ample bundle lying on the wall, see \cite{HuLi}.

However, starting in dimension three several fundamental problems appear:

(i) Although there are gauge-theoretic generalisations of the Donaldson-Uh\-len\-beck compactification to
higher-dimensional varieties due to Tian \cite{Tian00}, whose construction provides a major step towards the definition of new invariants for higher-dimensional (projective) manifolds as proposed by Donaldson and Thomas in their influential paper \cite{DonaldsonThomasGaugeTheory}, these are not known to possess a complex structure.

(ii) Adapting the notion of ``wall'' as in \cite{Qin93}, one immediately finds examples where these walls are
not locally finite inside the ample cone.

(iii) Looking at segments between two integral ample classes in the ample cone instead, Schmitt \cite{Sch00} 
gave examples of threefolds such that the point on the segment where the moduli space changes is no
longer rational (as in the case of surfaces) but is a non-rational class in the ample cone.

\subsection{Main results}In this paper we present and pursue a novel approach to attack and solve the above-mentioned problems. It is based on the philosophy  that the natural ``polarisations'' to consider when defining slope-semistability on higher-dimensional
base manifolds are not ample divisors but rather movable curves, cf.~\cite{Miyaoka, CaPe11}. 

Given an $n$-dimensional smooth projective variety $X$, we consider the open set $P(X) \subset H^{n-1,n-1}_{\mathbb{R}}(X)$ of powers $[H]^{n-1}$ of real ample divisor classes $[H] \in \mathrm{Amp}(X)$ inside the cone spanned by classes of movable curves. We prove that $P(X)$ is open in the movable cone, and that the natural map $\mathrm{Amp}(X) \to P(X)$ (taking $(n-1)$-st powers) is an isomorphism, see Proposition~\ref{P}. Moreover, we show that $P(X)$ supports a locally finite chamber structure given by linear rational walls such that the notion of slope-(semi)stability is constant within each chamber, see Theorem~\ref{thm:chambers}. Additionally, any chamber (even if it is not open) contains products  $H_1H_2...H_{n-1}$ of integral ample divisor classes, see Proposition~\ref{prop:a_cic_in_every_chamber}. These results explain and resolve the problem encountered by Qin, Schmitt, and others in their respective approaches to the wall-crossing problem.

By the results just discussed, we are thus led to the problem of constructing moduli spaces of torsion-free sheaves which are slope-semistable with respect to a \emph{multipolarisation} $(H_1,...,H_{n-1})$, where $H_1, ..., H_{n-1}$ are integral ample divisor classes on $X$.

Here and in the following, we say that a torsion-free sheaf $E$ on $X$ is \emph{slope-stable} (resp.~\emph{slope-semistable}) \emph{with respect to} $(H_1,...,H_{n-1})$, or $(H_1,...,H_{n-1})$-(semi)stable for short, if for any coherent subsheaf $F$ of intermediate rank in $E$ we have 
\[ \frac{c_1(F)\cdot H_1\cdot...\cdot H_{n-1}}{\mathrm{rk} (F)}< \  (\text{resp.} \ \le) \ \frac{c_1(E)\cdot H_1\cdot...\cdot H_{n-1}}{\mathrm{rk} (E)}.\] 

Using this terminology, we can now formulate our main result as follows.
\begin{mainthm}
Let $X$ be a projective manifold of dimension $n \geq 2$, $H_1,$ $...,$ $H_{n-1}$ $\in \mathrm{Pic}(X)$ ample divisors, 
$c_i \in H^{2i}\bigl(X, \mathbb{Z}\bigr)$ for $1\le i\le n$, 
 $r$ a positive integer, $c \in  K(X)_{num}$ a class with rank $r$ and Chern classes $c_j(c) = c_j$, 
and $\Lambda$ a line bundle on $X$ with $c_1(\Lambda)= c_1 \in H^2 (X, \mathbb{Z})$. 
Denote by $\underline{M}^{\mu ss}$ the functor that associates to each  weakly normal variety $B$ 
the set of isomorphism classes of $B$-flat families of $(H_1, ..., H_{n-1})$-semistable torsion-free coherent sheaves with class $c$ 
and determinant $\Lambda$ on $X$. Let $h_j$ be the class of $\mathscr{O}_{H_j}$ in the numerical Grothendieck group $K(X)_{num}$ of $X$, and $x \in X$ a point. Set
\[{u}_{n-1}(c) \definiere - r h_{n-1}\cdot...\cdot h_1 + \chi(c\cdot h_{n-1}\cdot...\cdot h_1)[\mathscr{O}_x] \in K(X)_{num}.\] 

Then, there exists a natural number $N \in \mathbb{N}^{>0}$, 
a weakly normal projective variety $M^{\mu ss} = M^{\mu ss}(c, \Lambda)$ together with an 
ample line bundle $\mathscr{O}_{M^{\mu ss}}(1)$, and a natural transformation 
$\underline{M}^{\mu ss} \to \underline{Hom}(\cdot , M^{\mu ss})$, mapping a family $\mathscr{E}$ to a classifying morphism $\Phi_\mathscr{E}$, with the following properties:
\begin{enumerate}
 \item For every $B$-flat family $\mathscr{E}$ of $(H_1, ..., H_{n-1})$-semistable sheaves of class $c$ and determinant $\Lambda$ with induced classifying morphism $\Phi_{\mathscr{E}}: B \to M^{\mu ss}$, we have
$$\Phi_{\mathscr{E}}^* \bigl(\mathscr{O}_{M^{\mu ss}}(1)\bigr) = \lambda_{\mathscr{E}}\bigl({u}_{n-1}(c) \bigr)^{\otimes N},$$
where $\lambda_\mathscr{E}\bigl({u}_{n-1}(c)\bigr)$ is the determinant line bundle on $S$ induced by $\mathscr{E}$ and ${u}_{n-1}(c)$. 
\item For any other triple $(M', \cO_{M'}(1), N')$ consisting of a projective variety $M'$,  an ample line bundle $\cO_{M'}(1)$ on $M'$, and a natural number $N'$ fulfilling the conditions spelled out in (1), 
we have $N|N'$, and 
there exists a uniquely determined morphism $\psi\colon M^{\mu ss} \to M'$ 
such that 
$\psi^*\bigl(\cO_{M'}(1)\bigr) \cong \cO_{M^{\mu ss}}(N'/N)$.
\end{enumerate}

The triple $(M^{\mu ss}, \cO_{M^{\mu ss}}(1), N)$  is uniquely determined up to isomorphism by the properties (1) and (2).
\end{mainthm}
The restriction to families over weakly normal parameter spaces $B$ is necessary due to the use of extension properties of weakly normal varieties in various crucial steps of our proof (see the discussion of ``Methods employed in the proof'' below). It is obscure at this point if and to what extent this restriction may be relaxed; weak normality is customarily used in the construction of the Chow variety, cf. \cite[Theorem I.3.21]{KollarRatCurves}, and special instances of our moduli spaces are isomorphic (or closely related) to Chow varieties of cycles of codimension two, see the discussion in \cite[Sect.~3.3]{GRT15}. We want to emphasise that a future generalisation to general parameter spaces $B$ will not change the analytification of $M^{\mu ss}$ as a topological space, but only its structure sheaf; cf.~the discussion of weakly normal complex spaces in Section~\ref{subsect:weaklynormal} below. 

In addition to the Main Theorem, we obtain the following results concerning the geometry of $M^{\mu ss}$: 
Two slope-semistable sheaves $F_1$ and $F_2$ give rise to different points in the moduli space $M^{\mu ss}$ 
if the double duals of the graded sheaves associated with Jordan-H\"older filtrations of $F_1$ and $F_2$, respectively, 
or the naturally associated two-codimensional cycles differ, see Theorem~\ref{thm:sep}. 
As a consequence, we conclude that $M^{\mu ss}$ 
contains the weak normalisation of the moduli space of (isomorphism classes of) $(H_1,...,H_{n-1})$-stable reflexive sheaves 
with the chosen topological invariants and determinant line bundle as a Zariski-open set, see Theorem~\ref{thm:compactificationofsimple}. In particular, it compactifies the moduli space of $(H_1,...,H_{n-1})$-stable vector bundles with the given invariants, thus answering in our particular setup an open question raised for example by Teleman~\cite[Sect.~3.2]{Tel08}; see Remark~\ref{rem:answeringTelemansQuestion}. 

Based on these results and on the study of examples such as those in \cite[Sect.~3.3]{GRT15}, we expect 
that the moduli space $M^{\mu ss}$ realises the following equivalence relation on the set of isomorphism classes of slope-semistable torsion-free sheaves: 
Two slope-semistable sheaves $F_1$ and $F_2$ should give rise to the same point in the moduli space $M^{\mu ss}$ 
if and only if the double duals of the graded sheaves associated with the respective Jordan-H\"older filtrations of $F_1$ and $F_2$ are the same and $F_1$, $F_2$ sit in the same connected component of a natural morphism of ``$\mathrm{Hilb}$-to-$\mathrm{Chow}$''-type. In Proposition \ref{prop:non-separation} we show that this is true when this morphism has  connected fibres.

 Comparing with the description of the geometry of the known topological compactifications of the moduli space of slope-stable vector bundles constructed by Tian \cite{Tian00} using gauge theory, we expect that the moduli spaces $M^{\mu ss}$ will provide new insight into the question whether these higher-dimensional analogues of the Donaldson-Uhlenbeck compactification admit complex or even projective-algebraic structures.

\subsection{Methods employed in the proof}
The proof of the Main Theorem follows ideas of Le Potier \cite{LePotierDonaldsonUhlenbeck} and Jun Li \cite{JunLiDonaldsonUhlenbeck} in the two-dimensional case; see also~\cite[Chap.~8]{HL} for a very nice account of these methods: first, using boundedness we parametrise slope-semistable sheaves by a locally closed subscheme $R^{\mu ss}$ of a suitable Quot-scheme. Isomorphism classes of slope-semistable sheaves correspond to orbits of a special linear group $G$ in $R^{\mu ss}$. 
We then consider a certain determinant line bundle $\mathscr{L}_{n-1}$ on $R^{\mu ss}$ and aim to show that it is generated by $G$-invariant global sections. Le Potier mentions in \cite[first lines of Sec.~4]{LePotierDonaldsonUhlenbeck} that in the case when $H_1=...=H_{n-1}=:H$ his proof of this fact in the two-dimensional case could be extended to higher dimensions, if a restriction theorem of Mehta-Ramanathan type were available for Gieseker-$H$-semistable sheaves. Indeed, such a result would be needed if one proceeded by restrictions to 
hyperplane sections on $X$. 
We avoid this Gieseker-semistability issue and instead restrict our families directly to the corresponding complete intersection curves, 
where slope-semistability and Gieseker-semistability coincide. 
The price to pay is some loss of flatness for the restricted families. This is the point at which our restriction to weakly normal parameter spaces comes into play: in order to overcome the difficulty created by the lacking flatness of restricted families, we pass to weak normalisations and show that sections in powers of $\mathscr{L}_{n-1}$ extend continuously, and hence holomorphically, over the non-flat locus. The moduli space $M^{\mu ss}$ then arises as the Proj-scheme of a ring of $G$-invariant sections of powers of $\mathscr{L}_{n-1}$ over the weak normalisation of $R^{\mu ss}$. Afterwards, the universal properties are established using the $G$-equivariant geometry of $R^{\mu ss}$ and its weak normalisation. 

\subsection{Outline of the paper} Section~\ref{sect:prelim} contains definitions and basic properties concerning determinant line bundles and semistability with respect to movable curve classes, followed by a discussion of the properties of weakly normal spaces and the proof of an elementary but crucial extension result for sections of line bundles on weakly normal varieties. In Section~\ref{subsect:flatrestr} we discuss the restriction of flat families of semistable sheaves to complete intersection curves. The corresponding class computations in the respective Grothendieck groups are carried out in Section~\ref{subsect:classcomput}, and the central semiampleness result for equivariant determinant bundles on Quot-schemes is proven in Section~\ref{subsect:proof}. In Section~\ref{sect:construction} the moduli space for slope-semistable sheaves is defined, and its functorial properties are established. This is followed in Section~\ref{sect:geometry} by a discussion of the basic geometry of the newly constructed moduli spaces, in particular concerning the sepa\-ra\-tion properties of classifying maps, the relation to the moduli space of simple sheaves, and the comparison with the Gieseker-Maruyama moduli space in those cases where the latter 
exist. Coming back to the motivating question, the final Section~\ref{sect:wall-crossing} discusses wall-crossing in the light of the newly constructed moduli spaces.

\addtocontents{toc}{\SkipTocEntry}
\subsection*{Acknowledgements}\hspace{-2.5	mm}
The authors want to thank Arend Bayer, Iustin Coand\u a, Hubert Flenner, Alain Genestier, Daniel Huybrechts, Alex K\"uronya, Adrian Langer, Jun Li, Sebastian Neumann, Julius Ross, Georg Schumacher, and Andrei Teleman for fruitful discussions.

\addtocontents{toc}{\SkipTocEntry}
\subsection*{Funding and support}
The first named author gratefully acknowledges support by the DFG-Forschergruppe 790 ``Classification of algebraic surfaces and compact complex manifolds'', as well as by the ``Eliteprogramm f\"ur Postdoktorandinnen und Postdoktoranden'' of the Baden-W\"urttemberg-Stiftung. More\-over, he wants to thank the Institut \'Elie Cartan, Nancy, for support and hospitality during several visits. Large parts of the research presented here were done while the first named author was holding a position at the Institute of Mathematics of Albert-Ludwigs-Universit\"at Freiburg. He wishes to express his gratitude for the support he obtained from the Institute and for the wonderful working conditions. Some parts of the research presented here were carried out during visits of the second named author at the Mathematical Insitutes of the Universities of Freiburg, Bonn, and Bochum. He wishes to thank these institutes for support and hospitality.

\section{Preliminaries}\label{sect:prelim}
We work over the field of complex numbers. A separated reduced scheme of finite type over $\mathbb{C}$ will be called an \emph{algebraic variety}. We emphasise that we do not assume varieties to be irreducible. An irreducible smooth projective variety will be called \emph{projective manifold}.

\subsection{Grothendieck groups and determinants}\label{subsect:Grothendieckgroup}
Let $X$ be a smooth projective irreducible variety of dimension $n$. The Grothendieck group $K(X)=K_0(X)= K^0(X)$ of coherent sheaves on $X$ becomes a commutative ring with $1 =[\mathscr{O}_X]$ by putting
$$[F_1]\cdot[F_2] \definiere [F_1 \otimes F_2] $$
for locally free sheaves $F_1$ and $F_2$. Two classes $u$ and $u'$ in $K(X)$ will be called \emph{numerically equivalent}, denoted $u \equiv u'$, if their difference is contained in the radical of the quadratic form 
$$(a,b) \mapsto \chi(a\cdot b). $$ We set $K(X)_{num}\definiere K(X)/\equiv$.

For any Noetherian scheme $Z$, we let $K^0(Z)$ and $K_0(Z)$ be the abelian groups generated by locally free sheaves and coherent $\mathscr{O}_Z$-modules, respectively, with relations generated by short exact sequences. A projective morphism $f\colon Y \to Z$ induces a homomorphism $f_!\colon K_0(Y) \to K_0(Z)$ defined by 
\[f_![F] \definiere \sum_{\nu \geq 0} [R^\nu f_*F].\] 

Any flat family $\mathscr{E}$ of coherent sheaves on a projective manifold $X$ parametrised by a Noetherian scheme $S$ defines an element $[\mathscr{E}] \in K^0(S \times X)$, and as the projection $p\colon S \times X \to S$ is a smooth morphism, we have a well-defined homomorphism $p_!\colon K^0(S \times X) \to K^0 (S)$, cf.~\cite[Cor.~2.1.11]{HL}. Let $q\colon S \times X \to X$ denote the second projection.
\begin{defi}
We define $\lambda_\mathscr{E}\colon K(X) \to \mathrm{Pic}(S)$ to be the composition of the following homomorphisms:
$$K(X) \overset{q^*}{\longrightarrow} K^0(S \times X) \overset{\cdot[\mathscr{E}]}{\longrightarrow} K^0(S \times X) \overset{p_!}{\longrightarrow} K^0(S) \overset{\mathrm{det}}{\longrightarrow} \mathrm{Pic}(S).$$
\end{defi}
We refer the reader to \cite[Sect.~8.1 and 2.1]{HL} for more details and for basic properties of this construction.

\subsection{Semistability with respect to multipolarisations}\label{subsect:semistability}

Let $X$ be a projective manifold of dimension $n$.
Semistability of torsion-free sheaves on $X$ is classically defined with respect to a polarisation, which is an ample class $H$ in the algebraic geometric context or a K\"ahler class $\phi$ in the complex case.  Although for the discussion of Gieseker-stability the class $H$ is needed as such, only its $(n-1)$-st power appears in the definition of slope-semistability, cf.~\cite[Chap.~1]{HL}. For the latter, it is therefore reasonable  to consider classes of curves rather than of divisors as polarisations. This point of view has been introduced in \cite{Miyaoka}, and has later been extended to include a discussion of semistability with respect to movable curve classes \cite{CaPe11}. 
We will give here the general definition before we specialise to the case of complete intersection classes, which is central for this paper.

\begin{defi}
 A curve $C \subset X$ is called \emph{movable} if there exists an irreducible algebraic family of curves containing $C$ as a reduced member and dominating $X$. A class $\alpha$ in the space $N_1=N_1(X)_{\R}$ of $1$-cycles on $X$ modulo numerical equivalence is called movable if it lies in the closure of the convex cone generated in $N_1$ by classes of movable curves.
\end{defi}

\begin{defi} \label{def:semistability}
 Let $\alpha \in N_1$ be a movable class. Then, a coherent sheaf $E$ on $X$ is called \emph{(semi)stable with respect to $\alpha$} or simply  \emph{$\alpha$-(semi)stable} if it is torsion-free, and if additionally for any proper non-trivial coherent subsheaf $F$ of $E$ we have 
$$\mu_{\alpha}(F) \definiere \frac{[\det F] \cdot \alpha}{\mathrm{rk}(F)} < (\leq) \frac{[\det E] \cdot \alpha}{\mathrm{rk}(E)} = \mu_{\alpha}(E).$$
The quantity $\mu_{\alpha}(F)$ is called the \emph{slope} of $F$ with respect to $\alpha$. By replacing the class $\alpha$ by the class $[\omega]\in H^{n-1,n-1}(X)$  of a  positive form on $X$ in the above inequality, we obtain the notion of \emph{$[\omega]$-(semi)stability}. When $H$ is an ample class or when $\phi$ is a K\"ahler class on $X$, we will speak of \emph{$H$-semistability} or of \emph{$\phi$-semistability} meaning stability with respect to $H^{n-1}$ or to $\phi^{n-1}$, respectively. A system  $(H_1, ..., H_{n-1})$  of $n-1$ integral ample classes on $X$ will be called a \emph{multipolarisation}. We will call a coherent sheaf \emph{$(H_1, ..., H_{n-1})$-(semi)stable} if it is (semi)stable with respect to the complete intersection class $H_1H_2...H_{n-1}\in N_1$. Once a (multi)polarisation has been fixed, we will just speak of \emph{$\mu$-semistability} or \emph{slope-semistability}. A $\mu$-semistable sheaf will be called \emph{$\mu$-polystable} if it is a direct sum of $\mu$-stable sheaves.
\end{defi}
\begin{rem}
Note that the notion of slope-semistability does not change when we multiply a given movable class $\alpha \in N_1$ by a constant $t \in  \mathbb{Q}^{>0}$.
\end{rem}

We will need two basic properties of semistable sheaves with respect to multipolarisations: Boundedness and Semistable Restriction.

\begin{prop}[Boundedness]\label{prop:slopeboundedness} The set of coherent sheaves with fixed Chern classes on $X$ that are semistable with respect to a fixed multipolarisation is bounded.
\end{prop}
This may be proven exactly as the corresponding statement in the case of a single polarisation \cite[Thm.~4.2]{Langer}; see also Proposition \ref{boundedness}. The reader is referred to  \cite[Thm.~5.2 and Cor.~5.4]{Langer} for the proof of the following result.
\begin{prop}[Semistable Restriction Theorem]\label{prop:slopeMR} If the coherent sheaf $E$ is (semi)stable with respect to the multipolarisation 
$(H_1, ..., H_{n-1})$ then there is a positive threshold $k_0\in \N^{>0}$ depending only on the topological type of $E$ such that for any $k\ge k_0$ and any smooth divisor $D\in |kH_1|$ with $(gr^\mu E)|_D$ torsion-free one has that $E|_D$ is (semi)stable with respect to $(H_2|_D, ..., H_{n-1}|_D)$. Here $gr^\mu E$ denotes the graded sheaf associated to a Jordan-H\"older filtration of $E$ as in section \ref{subsect:separation}.
\end{prop}

\subsection{Extension of sections on weakly normal spaces}\label{subsect:weaklynormal}
We quickly recall some notions motivated by the first Riemann extension theorem. For more information see \cite[Appendix to Chap.~2]{Fischer} and \cite[Sect.~I.7]{KollarRatCurves}.
\begin{defi}\label{Defi:weaklynormalcomplexspace}
Let $X$ be a reduced complex space, and $U \subset X$ an open subset. A continuous function $f \colon U \to \mathbb{C}$ is called \emph{$c$-holomorphic} if its restriction $f|_{U_{\mathrm{reg}}}$ to the regular part of $U$ is holomorphic. This defines a sheaf $\widehat{\mathscr{O}}_X$ of $c$-holomorphic functions. A reduced complex space is called \emph{weakly normal} if $\widehat{\mathscr{O}}_X = \mathscr{O}_X$, 
i.e. if any $c$-holomorphic function is in fact holomorphic.
\end{defi}
\begin{defi}\label{Defi:weaklynormalvariety}
A variety $X$ is \emph{weakly normal} if and only if the associated (reduced) complex space $X^{an}$ is weakly normal in the sense of Definition~\ref{Defi:weaklynormalcomplexspace} above.
\end{defi}
Note that by \cite[Prop.~2.24]{LeahyVitulli} and \cite[Cor.~6.13]{GrecoTraverso}, a variety is weakly normal in the sense of Definition~\ref{Defi:weaklynormalvariety} above if and only if it is weakly normal in the sense of \cite[Def.~2.4]{LeahyVitulli} if and only if it is seminormal in the sense of \cite[Def.~1.2]{GrecoTraverso}.

Proofs of the assertions of the following proposition are contained in \cite[Proposition I.7.2.3]{KollarRatCurves} and in \cite[Sect.~2.30]{Fischer}.
\begin{propnot}\label{propnot:weaknormalisation} For any algebraic variety/reduced complex space $X$ there exists a \emph{weak normalisation}; i.e., a weakly normal algebraic variety/complex space $X^{wn}$ together with a finite, surjective map $\eta \colon X^{wn} \to X$ enjoying the following universal property:
 if $Y$ is any weakly normal algebraic variety/complex space together with a regular/holomorphic map $\psi\colon Y \to X$, there exists a uniquely determined regular/holomorphic map $\hat \psi\colon Y \to X^{wn}$ that fits into the following commutative diagram:
\[\begin{xymatrix}{
  & X^{wn} \ar[d]^\eta \\
Y \ar[r]^{\psi} \ar[ur]^{\hat \psi}& X  .
}
  \end{xymatrix}
\]
If $X$ is a separated scheme of finite type (or complex space), then by slight abuse of notation the reduced weakly normal scheme (or complex space, respectively) $\bigl(X_{red}\bigr)^{wn}$ will also be denoted by $X^{wn}$.
\end{propnot}
\begin{rem}
 For any complex space $X$, the normalisation map $\eta\colon X^{wn} \to X$ is a homeomorphism, see \cite[Sect.~2.29]{Fischer} for the proof. 
\end{rem}

\begin{Lem}[\cite{Fischer}, p.~123]\label{lem:characterisationOfWeaklyNormalFunctions}
 Let $X$ be a weakly normal complex space, and let $\pi\colon X' \to X$ denote its normalisation. Then, we have $\mathscr{O}_X = \pi_*(\mathscr{O}_{X'})^\pi$, i.e., the holomorphic functions on $X$ are exactly the holomorphic functions on $X'$ that are constant on the fibres of $\pi$.
\end{Lem}

The following elementary extension result result will play a crucial role throughout the paper. 
\begin{Lem}[Extension of sections on weakly normal varieties]\label{sections}
  Let $G$ be a connected algebraic group, let $S$ be a weakly normal $G$-variety, and let $\mathscr{L}$ be a $G$-linearised line bundle on $S$. Then, there exists a finite system of irreducible subvarieties $(S_i)_{i=1, \ldots, m}$ of $S$ with the following property: For any 
  closed $G$-invariant subvariety $T$ of $S$ such that 
\begin{enumerate}
 \item[(i)] the intersection of $T$ with each irreducible component $\Sigma$ of $S$ has codimension at least two in $\Sigma$, and
 \item[(ii)] $T$ contains none of the $S_i$,
\end{enumerate}
any $G$-invariant section $\sigma \in H^0\bigl(S \setminus T, \mathscr{L} \bigr)^G$ extends to a $G$-invariant section $\bar \sigma \in H^0\bigl(S, \mathscr{L} \bigr)^G$.
\end{Lem}
\vspace{-0.4cm}
\begin{proof} 
Let $\pi\colon S'\to S$ be the normalisation morphism of $S$. Note that $\pi \colon S' \to S$ is also the normalisation of $S$ in the analytic category. 
Let $B\subset S$ be the set of points of $S$ over which $\pi$ is not a local isomorphism. 
Consider the irreducible components $B'_j$ of $\pi^{-1}(B)$, their projections $B_j:= \pi (B'_j)$ on $S$, and let the family $(S_i)_{i=1, \ldots, m}$ consist  of the irreducible singular stratification of intersections of the type $B_{j_1}\cap B_{j_2}$. Here, by the irreducible singular stratification of a subvariety $\Sigma$ of $S$ we mean the set of irreducible components of the subvarieties $\Sigma$, $\mathrm{Sing}(\Sigma)$, $\mathrm{Sing}(\mathrm{Sing}(\Sigma))$, etc. 
  
To establish the desired extension, we first pull back $\sigma$ to a section $\sigma'$ in $H^0\bigl(S' \setminus \pi^{-1}(T), \mathscr{\pi^*L} \bigr)^G$. If $S = \bigcup_k \Sigma_k$ is the decomposition of $S$ into irreducible components, $S'$ decomposes into connected components as follows: $S' = \bigcup_k \Sigma_k'$ with $\pi(\Sigma'_k) = \Sigma_k$. Hence, $\pi^{-1}(T)$ intersects every connected component of $S'$ in codimension at least two, and the section $\sigma'$ extends to $S'$ by normality. We will show below that this section, which we continue to denote by $\sigma'$, is constant on the fibres of $\pi$. By Lemma~\ref{lem:characterisationOfWeaklyNormalFunctions}, $\sigma$ hence extends to a holomorphic section of $\mathscr{L}$ in a neighbourhood of $s$. Global \emph{regular} extension then follows from the algebraicity of the original section $\sigma$. Moreover, since $S\setminus T$ is dense in $S$, $G$-invariance of the extended section $\bar \sigma$ will follow directly from the $G$-invariance of the original section $\sigma$. 

First, we consider points $s \in T \setminus B$. By definition of $B$, this implies that $S$ is normal at $s$, and consequently, that the normalisation morphism $\pi$ is an isomorphism over a small neighbourhood $U$ of $s$ in $S$. In particular, the fibre of $\pi$ over $s$ consists of a single point, so there is nothing to show. Second, let $s\in B \cap T$. If $\pi^{-1}(s) \subset S'$ consists of a single point, again there is nothing to show. Suppose now that $s$ has at least two distinct preimages $s'_1,s'_2\in S'$.  We let $B'_1$ and $B'_2$ be irreducible components of $\pi^{-1}(B)$ passing through $s'_1$ and $s'_2$, respectively; $B'_1$ and $B'_2$ may coincide. Then, $s\in B_1\cap B_2$. We take $S_i$ in the singular stratification of $B_1\cap B_2$ that contains $s$ as a smooth point (of $S_i$). 
Then, for $j=1,2$ there exist irreducible subvarieties $S'_{i,j}\subset B'_j$ of $\pi^{-1}(S_i)$ that pass through $s'_1$ and $s'_2$, respectively, and that project onto $S_i$. We fix a trivialisation of $\mathscr{L}$ around $s$, which we also lift to a trivialisation of $\pi^*(\mathscr{L})$ around $s'_1$ and $s'_2$. We let $f_{\sigma'}$ be the holomorphic function that represents $\sigma'$ in the chosen trivialisation. As explained above, it suffices to show that $f_{\sigma'}$ is constant on fibres of $\pi$. Since $T$ does not contain $S_i$, there exists a sequence $(p_n)$ of points in $S_i\setminus T$ converging to $s$ in the classical topology. As $\pi$ is proper and finite, we may lift $(p_n)$ to two sequences $(p^{(1)}_n)$ and $(p^{(2)}_n)$ in $S'_{i,1}$ and $S'_{i,2}$, respectively, that converge to $s_1'$ and $s_2'$, respectively. The values of $\sigma$ on $(p_n)$ coincide with the values of $\sigma'$ on the lifted sequences $(p^{(1)}_n)$ and $(p^{(2)}_n)$. We thus have $f_{\sigma'} (s'_1) = \lim_{n \to \infty} \bigl(f_{\sigma'}(p^{(1)}_n)\bigr) = \lim_{n \to \infty} \bigl(f_{\sigma'}(p^{(2)}_n)\bigr)=f_{\sigma'} (s'_2) \in \mathbb{C}$. 
\end{proof}

\section{Semiampleness of determinant line bundles}\label{section:semiampleness}
In this core section of the paper, we prove the crucial semiampleness statement, Theorem~\ref{thm:semiampleness}, that will later allow us to define the desired moduli space as the Proj of an appropriate graded ring of sections. Our main idea is to produce sections in determinant line bundles by exploiting the existence of sections in determinant line bundles arising from families of sheaves over complete intersection curves.

Before Theorem~\ref{thm:semiampleness} is proven in Section~\ref{subsect:proof}, two technical issues are discussed in Sections \ref{subsect:flatrestr} and \ref{subsect:classcomput}: how does flatness fail precisely when a flat family is restricted to a complete intersection curve, and how does the determinant line bundles of a flatly restricted family compare with the determinant line bundle of the original family?
\subsection{Flat restriction to curves}\label{subsect:flatrestr}
  \begin{Lem}[Preserving flatness under restriction to hyperplane sections]\label{flatness3}
Let $G$ be a connected algebraic group. 
  Let $X$ be a projective manifold of dimension $n\ge 2$, $H$ a very ample polarisation on $X$, $S$ an algebraic $G$-variety,
 $S_1, \dots, S_m$ closed irreducible $G$-invariant subvarieties  
of $S$, 
and $\mathscr{E}$ a $G$-linearised $S$-flat family of coherent sheaves on $X$ such that for each $i \in \{1, \ldots, m\}$ there exists some point $s_i\in S_i$ such that $\mathscr{E}_{s_i}$ is torsion-free. 

Then, there exists a dense open subset $U$ in $|H|$ such that every $X' \in U$ is smooth, and such that for every $X' \in U$ there exists a $G$-invariant closed subvariety $T \subset S$ with the following properties:
\begin{enumerate}
 \item[(i)] The restriction $\mathscr{E}|_{S \times X'}$ is flat over $S\setminus T$.
 \item[(ii)] Each $\mathscr{E}_{s_i}|_{X'}$ is torsion-free on $X'$. 
 \item[(iii)] If we denote the restriction $\mathscr{E}|_{(S\setminus T)\times X'}$ by $\mathscr{E}'$, the following sequence of $G$-linearised $(S\setminus T)$-flat sheaves is exact:
\begin{align}
0\to \mathscr{E}|_{(S\setminus T) \times X}\bigl(-((S\setminus T)\times X') \bigr)&\to \mathscr{E}|_{(S\setminus T) \times X} \to \mathscr{E}'\to 0.
\end{align}
 \item[(iv)] For every $i \in \{1, \ldots, m\}$ the intersection $T \cap S_i$ has codimension at least two in $S_i$. 
\end{enumerate} 
\end{Lem}
\begin{proof}
The set $S_{\mathrm{tf}} \subset S$ parametrising torsion-free sheaves in the family $\mathscr{E}$ is $G$-invariant and Zariski-open in $S$. Moreover, for every $i \in \{1, \ldots, m\}$ the intersection of $S_{\mathrm{tf}}$ with $S_i$ is dense in $S_i$. For any (smooth) hyperplane section $X'\in |H|$, the restriction $\mathscr{E}|_{S_{\mathrm{tf}}\times X'}$ is flat, and the sequence
$$0\to \mathscr{E}|_{S_{\mathrm{tf}} \times X}(-(S_{\mathrm{tf}}\times X'))\to \mathscr{E}|_{S_{\mathrm{tf}} \times X}\to\mathscr{E}|_{S_{\mathrm{tf}}\times X'}\to 0$$ is exact, see \cite[Lemma 2.1.4]{HL}.

In addition to the given points $s_i$, choose further points, one on each irreducible component of $S_i\setminus S_{\mathrm{tf}}$, whenever these open subschemes are non-empty; call the resulting finite set $S_{\mathrm{ch}}$.
By \cite[Lemma 1.1.12]{HL}, for each $s \in S_{\mathrm{ch}}$ there exists a Zariski-open dense subset $V(s)$ of $|H|$ such that every $X'\in V(s)$ is smooth and  $\mathscr{E}_{s}$-regular, i.e., the natural map $\mathscr{E}_{s}(-X')\to \mathscr{E}_{s}$ is injective. Moreover, by \cite[Corollary 1.1.14.ii]{HL} the set $V(s)$ may be chosen in such a way that additionally for every $X' \in V(s)$ all the restricted sheaves $\mathcal{F}_{s_i}|_{X'}$ remain torsion-free. 

With these preparatory considerations in place, if we set 
$U \definiere \bigcap_{s \in S_{\mathrm{ch}}} V(s)$, it follows from \cite[Cor.~on p.177]{Mat} that for each $ X' \in U$ the restricted family $\mathscr{E}|_{S\times X'}$ is flat over a $G$-invariant open neighbourhood $N(X', s)$ of $s$ in $S$, and that the associated sequence
$$0\to \mathscr{E}|_{N(X', s) \times X}(-(N(X', s)\times X'))\to \mathscr{E}|_{N(X', s) \times X}\to\mathscr{E}|_{N(X', s)\times X'}\to 0$$
is exact. Therefore, the subvariety $T \definiere S \setminus \bigl( S_{\mathrm{tf}} \cup \bigcup_{s \in S_{\mathrm{ch}}} N(X', s) \bigr)$
has the desired properties (i) -- (iv).
\end{proof}

Repeated application of the preceding lemma yields the following result.
  \begin{Cor}[Preserving flatness under repeated restriction to hyperplane sections]\label{flatness2} Let $G$ be a connected algebraic group. 
  Let $X$ be a projective manifold of dimension $n\ge 2$, $H_1, H_2, \ldots, H_{n-1}$  very ample polarisations on $X$,$S$ an algebraic $G$-variety,
 $S_1, \dots, S_m$ closed irreducible $G$-invariant subvarieties  
of $S$, 
and $\mathscr{E}$ a $G$-linearised $S$-flat family of coherent sheaves on $X$ such that for each $i \in \{1, \ldots, m\}$ there exists some point $s_i\in S_i$ such that $\mathscr{E}_{s_i}$ is torsion-free. Then, there there exist dense open subsets $U_l \subset |H_l|$, $l = 1, \ldots, n-1$, such that for any choice of elements $X_l\in U_l$, $1\le l\le n-1$, the associated complete intersections $X^{(0)}\definiere X$ and $X^{(l)}\definiere\cap_{k=1}^lX_k$ are smooth, and such that for every such choice there exists a $G$-invariant closed subvariety $T \subset S$ with the following properties:
\begin{enumerate}
 \item[(i)] For all $1\le l\le n-1$, the restriction $\mathscr{E}|_{S \times X^{(l)}}$ is flat over $S\setminus T$.
 \item[(ii)] Each $\mathscr{E}_{s_i}|_{X^{(l)}}$ is torsion-free on $X^{(l)}$. 
 \item[(iii)] If we denote the restriction $\mathscr{E}|_{(S\setminus T)\times X^{(l)}}$ by $\mathscr{E}^{(l)}$, for all $1\le l\le n-1$ the following sequence of $G$-linearised $(S\setminus T)$-flat sheaves is exact:
\begin{align}
0\to \mathscr{E}^{(l-1)}\bigl(-((S\setminus T)\times X^{(l)})\bigr)&\to \mathscr{E}^{(l-1)} \to \mathscr{E}^{(l)}\to 0.
\end{align}
 \item[(iv)] For every $i \in \{1, \ldots, m\}$ the intersection $T \cap S_i$ has codimension at least two in $S_i$. 
\end{enumerate}
 \end{Cor}
\begin{rems}\label{rem:orbitsmaybeamongtheSi}
a) The assertion of Corollary~\ref{flatness2} should be compared with the assumptions of Lemma~\ref{sections}, 
as well as with the assumptions of Proposition~\ref{prop:determinantcomparison}. 

b) Given a finite number of points $p_1, \ldots, p_k \in S$, we may assume the $G$-invariant closed irreducible subvarieties $\overline{G\acts p_1}, \ldots, \overline{G\acts p_k}$ to be among the $S_i$. The resulting closed subvariety $T$ will then have empty intersection with the set $\{p_1, \ldots, p_k\}$; i.e., none of the $p_i$ is contained in $T$. 

c) If $S$ parametrises a family of torsion-free sheaves (e.g., slope-semi\-stable sheaves) we may choose the irreducible components of $S$ to be among the $S_i$.
\end{rems}

\subsection{Class computations}\label{subsect:classcomput}
Here we extend some class computations from the surface case \cite[Sect.~8.2]{HL} to the case of $n$-dimensional projective manifolds, $n\ge 2$. 
We remark that this generalisation will work for multipolarisations of type $(H_1, \ldots, H_{n-1})$ where the ample divisors $H_1, \ldots,H_{n-1}$ may differ. 

We will denote the degree of $X$ with respect to $H_1,\ldots, H_{n-1}, H_i$ , i.e., with $H_i$ appearing twice, by $d_{i}\definiere H_1\cdot \ldots \cdot H_{n-1}\cdot H_i$.
\subsubsection{Setup and notation}\label{subsubsection:SetupClasses}
Let $X$ be a projective $n$-dimensional  manifold, $H_1, \ldots, H_{n-1} \in \mathrm{Pic}(X)$ 
 ample divisors, $c_i \in H^{2i}\bigl(X, \mathbb{Z}\bigr)$, $1\le i\le n$, classes on $X$, $r$ a positive integer,
 $c \in  K(X)_{num}$ a class with rank $r$ and Chern classes $c_j(c) = c_j$, 
and $\Lambda \in \mathrm{Pic}(X)$ a line bundle with $c_1(\Lambda)= c_1 \in H^2 (X, \mathbb{Z})$. 

 We denote by $h_i$ the class of $\mathscr{O}_{H_i}$ in $K(X)$, which coincides with $[\mathscr{O}_X]-[\mathscr{O}_X(-H_i)]$. 
For the following definition we shall suppose that the divisors  $H_1, \ldots H_{n-1}$ are very ample. 
(It will be clear from our considerations that this does not constitute a restriction of generality.) Given general elements $X_i \in |H_i|$, $1\leq i \leq n-1$, put 
\[X^{(0)} \definiere X, \,\text{and } X^{(l)}\definiere \textstyle{\bigcap}^{l}_{i=1}X_i \text{ for }1\leq l \leq n-1. \]
We will also write $X'$ instead of $X^{(1)}$. Choose a fixed base point $x \in X^{(n-1)}$, set 
 \begin{align*}
 u_0(c|_{X^{(n-1)}})&\definiere - r [\mathscr{O}_{X^{(n-1)}}] + \chi(c|_{X^{(n-1)}})[\mathscr{O}_x] \in K(X^{(n-1)}),
\intertext{for $1\le i\le n-2$ set}
 u_i(c|_{X^{(n-1-i)}})&\definiere - r h_{n-1}|_{X^{(n-1-i)}}\cdot \ldots\cdot h_{n-i}|_{X^{(n-1-i)}}\,\, +\\     
 &\quad \quad +\,\chi(c|_{X^{(n-1-i)}}\cdot h_{n-1}|_{X^{(n-1-i)}}\cdot \ldots\cdot h_{n-i}|_{X^{(n-1-i)}})[\mathscr{O}_x]\\
                    &\in K\bigl(X^{(n-1-i)}\bigr),
\intertext{and finally, set}
 {u}_{n-1}(c) &\definiere  - r h_{n-1}\cdot\ldots\cdot h_1 + \chi(c\cdot h_{n-1}\cdot\ldots\cdot h_1)[\mathscr{O}_x] \in K(X).
\end{align*}

Note that the definition of the class $u_{n-1}(c)$ does not require restrictions to hyperplane sections. We can therefore use this definition also when the divisors $H_1$, \ldots,$H_{n-1}$ are only supposed to be ample. We will stress the dependence on $H_1$, \ldots,$H_{n-1}$ by writing
$u_{n-1}(c; H_1,\ldots,H_{n-1})$ instead of just $u_{n-1}(c)$. 

Let now $S$ be a scheme of finite type over $\C$ and $\mathscr{E}$ an $S$-flat family of coherent sheaves on $X$ with class $c$ and fixed determinant bundle line $\Lambda$. It is explained in \cite[Example 8.1.8.ii]{HL} that for two numerically equivalent classes $D_0,D_1\in K(X)$ of zero-dimensional sheaves on $X$ one obtains isomorphic determinant line bundles  
$\lambda_{\mathscr{E}}( D_0)$, $\lambda_{\mathscr{E}}( D_1)$ on $S$; see also \cite[Prop.~3.2]{LePotierDonaldsonUhlenbeck}.
In particular, the determinant line bundles $\lambda_{\mathscr{E}}\bigl( u_i(c|_{X^{(n-1-i)}})\bigr)$ will not depend on the choice of the point $x\in X^{(n-1)}$. Using the same reasoning and the fact that $[\mathscr{O}_{a_iH_i}]=1-(1-h_i)^{a_i}$, we find for any positive integers $a_i$, $1\le i\le n-1$, that 
\begin{equation*}
 \lambda_{\mathscr{E}}\bigl( u_{n-1}(c; a_1H_1,\ldots,a_{n-1}H_{n-1}) \bigr)\cong\lambda_{\mathscr{E}}\bigl( u_{n-1}(c; H_1,\ldots,H_{n-1})\bigr)^{\otimes a_1\cdots a_{n-1}}.
\end{equation*}
In fact, $u_{n-1}(c; H_1,\ldots,H_{n-1})$ only depends on $c$ and on the class $\alpha\definiere h_{n-1}\cdot\ldots\cdot h_1 \in K(X)$ of the complete intersection curve $X^{(n-1)}$. 

Numerical equivalence for classes in $K(X)$ will be denoted by $\equiv$.

\subsubsection{Determinant line bundles of restricted families}
The following result compares determinant line bundles of flat families of sheaves on $X$ 
with determinant line bundles of flat families of restricted sheaves on the curve $X^{(n-1)}$. 
It generalises the computations in the surface case, see p.~223 of \cite{HL}. 
We will use the notation introduced in the previous paragraphs as well as the terminology of Section~\ref{subsect:Grothendieckgroup}:

\begin{prop}[Determinant line bundles of restricted families]\label{prop:determinantcomparison} Let $G$ be a connected algebraic group and $S$ be an algebraic $G$-variety. 
Let $\mathscr{E}$ be a $G$-linearised $S$-flat family of torsion-free sheaves on $X$ with class $c$ and fixed determinant line bundle $\Lambda$ such that for all $i \in \{1, \ldots, n-1\}$
\begin{enumerate}
  \item[(i)] the restriction $\mathscr{E}^{(i)}\definiere \mathscr{E}|_{S \times X^{(i)}}$ is flat over $S$, and
 \item[(ii)] the sequence
\begin{equation}\label{eq:exactsequence}
 0\to \mathscr{E}^{(i-1)}\bigl(-X^{(i)}\bigr)\to \mathscr{E}^{(i-1)} \to \mathscr{E}^{(i)}\to 0 
\end{equation}
is exact.
\end{enumerate}
Then, for all $i \in \{1, \ldots, n-1\}$, there exist an isomorphism \begin{equation}\label{eq:intermediatestep}
\lambda_{\mathscr{E}^{(n-i)}}\bigl({u}_{i-1}(c|_{X^{(n-i)}})\bigr)^{\otimes d_{n-i}}\cong
\lambda_{\mathscr{E}^{(n-i-1)}}\bigl(u_{i}(c|_{X^{(n-i-1)}})\bigr)^{\otimes d_{n-i}}
\end{equation}
of $G$-linearised line bundles on $S$. In particular, there exists an isomorphism \begin{equation}\label{eq:determinantcomparison}
\lambda_{\mathscr{E}^{(n-1)}}\bigl({u}_0(c|_{X^{(n-1)}})\bigr)^{\otimes d_1d_2\cdots d_{n-1}}\cong\lambda_{\mathscr{E}}\bigl(u_{n-1}(c)\bigr)^{\otimes d_1d_2\cdots d_{n-1}}
\end{equation}
of $G$-linearised line bundles on $S$.

\end{prop}

\begin{proof}
Recall that the determinant line bundle associated to a $G$-linearised flat family over $S$ is also $G$-linearised. 
One can also check that a short exact sequence $0\to\cE'\to\cE\to\cE''\to0$ 
of $G$-linearised flat families over $S$ induce an isomorphism $\lambda_{\cE}\cong\lambda_{\cE'}\otimes\lambda_{\cE''}$ of $G$-linearised determinant line bundles, cf.
\cite[Lemma 8.1.2]{HL}.
Therefore the following considerations will be compatible with the group action.

In order to prove our statements, it will be enough 
to show for that for $1\le i\le n-1$ an isomorphism as in equation \eqref{eq:intermediatestep} exists. It is clear that for this it suffices to consider the special case $i=n-1$, where $n\ge 2$ is arbitrary. Setting $\mathscr{E}' \definiere  \mathscr{E}^{(1)}$, we thus need to check that 
\begin{equation}
\lambda_{\mathscr{E}'}\bigl({u}_{n-2}(c|_{X'})\bigr)^{\otimes d_{1}}\cong
\lambda_{\mathscr{E}}\bigl(u_{n-1}(c)\bigr)^{\otimes d_{1}}.
\end{equation}

In order to do this, we will use the following auxiliary class in $K(X)$:
\[ w \definiere - \chi(c\cdot h_{n-1}\cdot\ldots\cdot h_1\cdot [\mathscr{O}_{X'}])h_{n-1}\cdot\ldots\cdot  h_2 + \chi(c\cdot h_{n-1}\cdot\ldots\cdot  h_2\cdot[\mathscr{O}_{X'}])h_{n-1}\cdot\ldots\cdot  h_1  .\]
 
The idea of the proof is to compute the restriction of $w$ to $X'$ in two different ways, cf.~\cite[p.~223]{HL}. Indeed, on the one hand, we may write
\begin{align}\label{eq:firstway}
\begin{split}
w|_{X'} &= -\chi(c|_{X'}\cdot h_{n-1}|_{X'}\cdot \ldots\cdot h_1|_{X'})h_{n-1}|_{X'}\cdot \ldots\cdot h_2|_{X'} \, \, + \\ &\quad \quad \quad \quad \quad +\, \chi(c|_{X'}\cdot h_{n-1}|_{X'}\cdot \ldots\cdot h_2|_{X'})h_{n-1}|_{X'}\cdot \ldots\cdot h_1|_{X'} \\
             &\equiv d_{1}(-r h_{n-1}|_{X'}\cdot \ldots\cdot   h_2|_{X'}+\chi(c|_{X'}\cdot h_{n-1}|_{X'}\cdot \ldots\cdot  h_2|_{X'})[\mathscr{O}_x])\\
             &=d_{1}\,{u}_{n-2}(c|_{X'}) \in K(X').
\end{split}
\end{align}
On the other hand, we have 
\begin{align}\label{eq:secondway}
\begin{split}
 w- w\cdot [\mathscr{O}_X(-X')]  
&=w\cdot[\mathscr{O}_{X'}]\\
&=w\cdot h_1\\
                               &= - \chi(c\cdot  h_{n-1}\cdot \ldots\cdot h_1\cdot h_1)  h_{n-1}\cdot \ldots\cdot h_1\,\, + \\
& \quad \quad \quad \quad +\, \chi(c\cdot  h_{n-1}\cdot \ldots\cdot h_1) h_{n-1}\cdot \ldots\cdot h_1\cdot  h_1\\
			      &\equiv d_1(-r h_{n-1}\cdot \ldots\cdot h_1+
\chi(c\cdot  h_{n-1}\cdot \ldots\cdot h_1)[\mathscr{O}_x])\\
			      &= d_1u_{n-1}(c).
\end{split}
\end{align}

Since $\mathscr{E}$ is $S$-flat with fixed determinant line bundle and since it has the additional properties that the restriction $\mathscr{E}'\definiere \mathscr{E}|_{S \times X'}$ remains flat over $S$, and that sequence \eqref{eq:exactsequence}, which in our current notation reads
\begin{equation}\label{eq:specialcase}
  0\to \mathscr{E}(-X')\to \mathscr{E} \to \mathscr{E}'\to 0,
\end{equation}
is exact, we obtain the following equivariant isomorphisms of $G$-linearised determinant line bundles on $S$:
\begin{align}\label{eq:ersteIsomorphie}
\lambda_{\mathscr{E}'}\bigl({u}_{n-2}(c|_{X'})\bigr)^{\otimes d_1}&\cong\lambda_{\mathscr{E}'}(w|_{X'}) && \text{ by \eqref{eq:firstway}} \notag\\
&=\lambda_{\mathscr{E}}(w)\otimes \lambda_{\mathscr{E}(-X')}(w)^{-1} && \text{ by \eqref{eq:specialcase}} \notag \\
&\cong \lambda_{\mathscr{E}}\bigl(w-w\cdot [\mathscr{O}_X(-X')]\bigr) && \\
&\cong \lambda_{\mathscr{E}}\bigl( {u}_{n-1}(c)\bigr)^{\otimes d_{1}} && \text{ by \eqref{eq:secondway}}\notag.
\end{align}
As noted above, this completes the proof of Proposition~\ref{prop:determinantcomparison}.
\end{proof}

\subsection{Semiampleness of determinant line bundles} \label{subsect:proof}
In this section we prove the crucial semiampleness statement, which will later allow us to define the desired moduli space as the Proj-scheme associated with some finitely generated ring of invariant sections.

\subsubsection{Setup}\label{setup}
Let $X$ be a projective $n$-dimensional  manifold, $H_1, \ldots , H_{n-1} \in \mathrm{Pic}(X)$ 
 ample divisors, $c_i \in H^{2i}\bigl(X, \mathbb{Z}\bigr)$, $1\le i\le n$, classes on $X$, $r$ a positive integer, 
 $c \in  K(X)_{num}$ a class with rank $r$ and Chern classes $c_j(c) = c_j$, 
and $\Lambda \in \mathrm{Pic}(X)$ a line bundle with $c_1(\Lambda)= c_1 \in H^2 (X, \mathbb{Z})$. 
By \emph{$\mu$-semistable} we always mean slope-semistable with respect to the multipolarisation $(H_1, \ldots , H_{n-1})$, cf.~the discussion in Section~\ref{subsect:semistability}.

Recall from Proposition~\ref{prop:slopeboundedness} that the family of $\mu$-semistable sheaves of class $c$ (and determinant $\Lambda$) is bounded, 
so that for sufficiently large $m \in \mathbb{N}$, each $\mu$-semistable sheaf  of class $c$ is $m$-regular with respect to some chosen ample line bundle $\mathscr{O}_X(1)$, cf.~\cite[Lem.~1.7.2]{HL}. 
In particular, for each such sheaf $F$, the $m$-th twist $F(m)$ is globally generated with $h^0(F(m)) = P(m)$, where $P$ is the Hilbert polynomial of $F$ with respect to $\cO_X(1)$, see for example \cite[Thm.~1.8.3]{Lazarsfeld}.
Setting $V \definiere \mathbb{C}^{P(m)}$ and $\mathscr{H}\definiere V \otimes \mathscr{O}_X(-m)$,  we obtain a surjection
$\rho\colon \mathscr{H} \to F$ by composing the evaluation map $H^0\bigl(F(m) \bigr) \otimes \mathscr{O}_X(-m)$ with an isomorphism $V \to H^0\bigl(F(m) \bigr)$. 
The sheaf morphism $\rho$ defines a closed point
$$[q\colon \mathscr{H} \to F] \in \mathrm{Quot}(\mathscr{H}, P)$$
in the Quot-scheme of quotients of $\mathscr{H}$ with Hilbert polynomial $P$.

Let $R^{\mu ss} \subset \mathrm{Quot}(\mathscr{H}, P)$ be the locally closed subscheme of all quotients $[q\colon \mathscr{H} \to F]$ with class $c$ and determinant $\Lambda$ such that 
\begin{enumerate}
 \item[(i)] $F$ is $\mu$-semistable of rank $r$, and
 \item[(ii)] $\rho$ induces an isomorphism $V \overset{\cong}{\longrightarrow} H^0\bigl(F(m) \bigr)$.
\end{enumerate}

The reductive group $\SL(V)$ acts on $R^{\mu ss}$ by change of base in the vector space $H^0\bigl(F(m) \bigr)$. This group action can be lifted to the reduction $R^{\mu ss}_{red}$, making the reduction morphism $R^{\mu ss}_{red} \to R^{\mu ss}$ equivariant with respect to the two $\SL(V)$-actions. Lifting the action one step further, we note that $\SL(V)$ also acts on the weak normalisation 
\begin{equation}\label{eq:Sdefined}
S\definiere  (R^{\mu ss}_{red})^{wn}
\end{equation}
of $R^{\mu ss}_{red}$ in such a way that the weak normalisation morphism
$(R^{\mu ss}_{red})^{wn} \to R^{\mu ss}_{red}$ intertwines the two $\SL(V)$-actions.

Let 
$\rho\colon  \mathscr{O}_{S} \otimes \mathscr{H} \to \mathscr{F} $
denote the pullback of the universal quotient from $\mathrm{Quot}(\mathscr{H}, P(m))$ to $S$. Choosing a fixed base point $x \in X$, as in Section~\ref{subsubsection:SetupClasses} we consider the class
${u}_{n-1}(c) \definiere  - r h_{n-1}\cdot\ldots\cdot h_1 + \chi(c\cdot h_{n-1}\cdot\ldots\cdot h_1)[\mathscr{O}_x] \in K(X)$, 
and the corresponding determinant line bundle
\begin{equation}\label{eq:L2definition}
\mathscr{L}_{n-1}\definiere  \lambda_{\mathscr{F}} \bigl( u_{n-1}(c) \bigr)
\end{equation}
on the parameter space $S$. 
\begin{rem}
 Since by assumption all the sheaves parametrised by $S$ have the same determinant $\Lambda$, 
it follows from the argument in \cite[Ex.~8.1.8 ii)]{HL} that $\mathscr{L}_{n-1}$ is in fact independent of the chosen point $x \in X$, 
i.e., it is naturally induced by the classes of the two ample divisors $H_1$, $H_2$,\ldots ,$H_{n-1}$; see also the discussion in Section~\ref{subsubsection:SetupClasses} above.
\end{rem}
\subsubsection{Semiampleness}The following is the main result of this section and the core ingredient in the construction of the moduli space.

\begin{thm}[Equivariant semiampleness]\label{thm:semiampleness}
There exists a positive integer $\nu \in \mathbb{N}$ such that $\mathscr{L}_{n-1}^{\otimes \nu}$ is generated over $S$ by $\SL(V)$-invariant sections. 
\end{thm}
\begin{proof} Let $s \in S$ be a given point in $S$, which we will fix for the rest of the proof. We will show that there exists an invariant section in some tensor power of $\mathscr{L}_{n-1}$ that does not vanish at $s$. The claim then follows by Noetherian induction.

 By the Semistable Restriction Theorem, Proposition~\ref{prop:slopeMR}, there exist positive natural numbers $a_1, a_2, \ldots , a_{n-1} \in \mathbb{N}$ such that 
\begin{enumerate}
\item[(i)] $a_1H_1, a_2H_2, \ldots , a_{n-1}H_{n-1}$ are very ample, and 
\item[(ii)] for any general (smooth) complete intersection curve $X^{(n-1)}$ of elements in $|a_1 H_1|, |a_2 H_2|,\ldots ,|a_{n-1}H_{n-1}|$
 the restriction $\mathscr{F}_s|_{X^{(n-1)}}$ is semistable.
\end{enumerate}
\begin{rem}\label{rem:semistabilityoncurve}
Concerning the second point note that on the curve $X^{(n-1)}$ the notion of ``semistability'' is well-defined without fixing a further parameter.
\end{rem}

Hence, it follows from Corollary~\ref{flatness2} (applied to the very ample line bundles $a_1H_1$, \ldots , $a_{n-1}H_{n-1}$) and Lemma~\ref{sections}
 that in order to prove our claim without loss of generality we may assume to be in the following setup: 
\begin{Setup}
There exists a complete intersection curve $X^{(n-1)}$ obtained by intersecting general members $X_i$ of $|a_iH_i|$ such that the following holds: if we set $\mathscr{F}^{(l)} \definiere  \mathscr{F}|_{S \times X^{(l)}}$, then for all $l\in \{1, \ldots, n-1\}$ 
\begin{enumerate}
 \item[(i)] the family $\mathscr{F}^{(l)}$ is $S$-flat, and 
 \item[(ii)] the sequence
\begin{align}
0\to \mathscr{F}^{(l-1)}(-(S\times X^{(l)}))&\to \mathscr{F}^{(l-1)} \to \mathscr{F}^{(l)}\to 0.
 \end{align}
of $S$-flat sheaves is exact.
\end{enumerate}
\end{Setup}
For this, we have replaced $S$ by $S\setminus T$, where $T$ is the ``non-flatness'' locus from Corollary \ref{flatness2},  
and we note that we may assume $s \notin T$, cf.~Remark~\ref{rem:orbitsmaybeamongtheSi}.

For any class $c \in K(X)_{num}$ let $c^{(n-1)} = \imath_{X^{(n-1)}}^*(c)$ denote the restriction to $X^{(n-1)}$, 
and let $P^{(n-1)}= P(c^{(n-1)})$ be the associated Hilbert polynomial with respect to the ample line bundle $\mathscr{O}_{X^{(n-1)}}(1) \definiere \mathscr{O}_X(1)|_{X^{(n-1)}}$.  Let $m^{(n-1)}$ be a large positive integer, set $V^{(n-1)} \definiere \mathbb{C}^{P^{(n-1)}(m^{(n-1)})}$, $\mathscr{H}^{(n-1)}$ $\definiere V^{(n-1)}\otimes \mathscr{O}_{X^{(n-1)}}(-m^{(n-1)})$, and let $Q_{X^{(n-1)}}$ be the closed subscheme of $\mathrm{Quot}(\mathscr{H}^{(n-1)}, P^{(n-1)})$ 
parametrising quotients with determinant $\Lambda|_{X^{(n-1)}}$. Denote by  $\mathscr{O}_{Q_{X^{(n-1)}}} \otimes \mathscr{H}^{(n-1)} \to \widehat{\mathscr{F}}^{(n-1)} $ the universal quotient sheaf on $Q_{X^{(n-1)}}$, and set
\begin{equation}\label{eq:L0definition}
 \mathscr{L}_0^{(n-1)} \definiere  \lambda_{\widehat{\mathscr{F}}^{(n-1)}}\bigl(u_0(c^{(n-1)})\bigr).
\end{equation}

Note that on the curve $X^{(n-1)}$ slope-semistability and Gieseker-semistability coincide, cf.~Remark~\ref{rem:semistabilityoncurve}. As in \cite[p.~223]{HL}, the following lemma thus
 follows from the construction of the moduli space for semistable sheaves on the curve $X^{(n-1)}$.
\begin{lemma}\label{lem:GIT=semistable}
After increasing $m^{(n-1)}$ if necessary, the following holds: 
\begin{enumerate}
\item For a given point $[q: \mathscr{H}^{(n-1)} \to E] \in Q_{X^{(n-1)}}$ the following assertions are equivalent:
\begin{enumerate}
 \item[(i)] The quotient $E$ is a semistable sheaf and the induced map $V^{(n-1)} \to H^0(X, E(m^{(n-1)}))$ is an isomorphism.
 \item[(ii)] The point $[q] \in Q_{X^{(n-1)}}$ is GIT-semistable with respect to the natural $\SL(V^{(n-1)})$-linearisation of $\mathscr{L}_0^{(n-1)}$.
 \item[(iii)] There exists a positive integer $\nu \in \mathbb{N}$ and an $\SL(V^{(n-1)})$-invariant section 
$ \sigma \in H^0\bigl(Q_{X^{(n-1)}}, \mathscr{L}_0^{(n-1)}\bigr)$ such that $\sigma([q]) \neq 0$. 
\end{enumerate}
\item Two points $[q_i\colon \mathscr{H}^{(n-1)} \to E_i] \in Q_{X^{(n-1)}}$, $i=1,2$, are separated by invariant sections in some tensor power of  $\mathscr{L}_0^{(n-1)}$ if and only if either one of them is semistable but the other is not, or both points are semistable but $E_1$ and $E_2$ are not $S$-equivalent.
\end{enumerate}
\end{lemma}
In addition to assertions (1) and (2) of the previous lemma, by increasing $m^{(n-1)}$ further if necessary, we may assume that for each $s \in S$ the restricted sheaf $\mathscr{F}^{(n-1)}_s$ is $m^{(n-1)}$-regular with respect to $\mathscr{O}_{X^{(n-1)}}(1)$. Consequently, each such sheaf is globally generated and defines a closed point in $Q_{X^{(n-1)}}$ with the additional property that the induced map $V^{(n-1)} \to H^0\bigl(X^{(n-1)}, \mathscr{F}^{(n-1)}_s(m^{(n-1)}) \bigr)$ is an isomorphism.

If we denote the projection from $S \times X^{(n-1)}$ to the first factor by $p$, 
the push-forward $p_*\bigl(\mathscr{F}^{(n-1)} (m^{(n-1)}) \bigr)$  
is a locally free 
$\SL(V)$-linearised  
$\mathscr{O}_{S}$-sheaf of rank $P^{(n-1)}(m^{(n-1)})$ on $S$. The associated $\SL(V)$-equivariant 
projective frame bundle $\pi \colon \widetilde{S} \to S$ parametrises a quotient
$$\mathscr{O}_{\widetilde{S}} \otimes \mathscr{H}^{(n-1)} \to \pi^*\mathscr{F}^{(n-1)} \otimes \mathscr{O}_\pi(1),$$
which induces an $\SL(V^{(n-1)})$-invariant morphism $\Phi^{(n-1)}\colon \widetilde S \to Q_{X^{(n-1)}}$ that is compatible with the $\SL(V)$-action on $\widetilde S$. We summarise our situation in the following diagram:
\begin{equation}
\label{diagr}
\begin{gathered}
\begin{xymatrix}{
\widetilde{S} \ar[rr]^>>>>>>>>{\Phi^{(n-1)}} \ar[d]_\pi &  & Q_{X^{(n-1)}} \\
S. & & }
  \end{xymatrix}
\end{gathered}
\end{equation}

By Proposition~\ref{prop:determinantcomparison}, especially by equation \eqref{eq:determinantcomparison}, there exists positive integers $k_0$ and $k_{n-1}$ such that
\begin{equation}\label{eq:shortdeterminantcomparison}
 \lambda_{\mathscr{F}^{(n-1)}}\bigl(u_0(c^{(n-1)})\bigr)^{\otimes k_0} \cong \lambda_{\mathscr{F}}\bigl(u_{n-1}(c) \bigr)^{\otimes k_{n-1}},
\end{equation}
cf.~the discussion at the end of Section~\ref{subsubsection:SetupClasses}.

With these preparations in place, we compute:
\begin{equation}\label{eq:longlinebundleiso}
\begin{aligned}
 &(\Phi^{(n-1)})^*(\mathscr{L}_0^{(n-1)})^{\otimes k_0} \\
 =\; &(\Phi^{(n-1)})^*(\lambda_{\widehat{\mathscr{F}}^{(n-1)}}(u_0(c^{(n-1)})))^{\otimes k_0}&&\text{by definition, see eq.~\eqref{eq:L0definition}}\\
			      \cong\; &\lambda_{\pi^*\mathscr{F}^{(n-1)} \otimes \mathscr{O}_\pi(1)}\bigl(u_0(c^{(n-1)})\bigr)^{\otimes k_0}&&\text{by \cite[Lem.~8.1.2 ii)]{HL}}\\
			      \cong\; &\lambda_{\pi^*\mathscr{F}^{(n-1)}}\bigl(u_0(c^{(n-1)})\bigr)^{\otimes k_0}&&\text{by \cite[Lem.~8.1.2 ii)]{HL}}\\
			      \cong\; &\pi^*\lambda_{\mathscr{F}^{(n-1)}}\bigl(u_0(c^{(n-1)})\bigr)^{\otimes k_0}&&\text{by \cite[Lem.~8.1.2 iv)]{HL}}\\
			      \cong\; &\pi^*\lambda_{\mathscr{F}}\bigl(u_{n-1}(c) \bigr)^{\otimes k_{n-1}}&&\text{by eq.~\eqref{eq:shortdeterminantcomparison}}\\
			      =\; &\pi^*(\mathscr{L}_{n-1})^{\otimes k_{n-1}} &&\text{by definition, see eq.~\eqref{eq:L2definition}}.
\end{aligned}
\end{equation}

Now, let $\sigma$ be an $\SL(V^{(n-1)})$-invariant section in $(\mathscr{L}_0^{(n-1)})^{\otimes \nu k_0}$ that does not vanish at a given point of the form $[q\colon \mathscr{H}^{(n-1)}_t \to \mathscr{F}_t|_{X^{(n-1)}}]$. Since $\Phi^{(n-1)}$ is $\SL(V)$-invariant, the pullback $(\Phi^{(n-1)})^*(\sigma)$ is an $\SL(V^{(n-1)})\times \SL(V)$-invariant section in $(\Phi^{(n-1)})^*(\mathscr{L}_0^{(n-1)})^{\otimes k_0} \cong \pi^*\lambda_{\mathscr{F}}(\widehat u_2)^{\otimes \nu k_{n-1}}$. Since $\pi$ is a good quotient of $\widetilde S$ by the $\SL(V^{(n-1)})$-action, the isomorphism \eqref{eq:longlinebundleiso} implies that $(\Phi^{(n-1)})^*(\sigma)$ descends to a section $l_{\mathscr{F}}(\sigma) \in H^0(S, (\mathscr{L}_2)^{\otimes \nu k_{n-1}})^{\SL(V)}$ that does not vanish at $t \in S$.

Finally, recall that we want to produce a section in $\mathscr{L}_{n-1}$ that does not vanish at our given point $s \in S$.
As $\mathscr{F}_s|_{X^{(n-1)}}$ is semistable, and as the induced map $V^{(n-1)} \to H^0\bigl(X^{(n-1)},\, \mathscr{F}^{(n-1)}_s(m^{(n-1)}) \bigr)$ is an isomorphism, owing to Lemma~\ref{lem:GIT=semistable}(1) there exists a positive integer $\nu \in \mathbb{N}$ and an invariant section $\sigma \in H^0\bigl(Q_{X^{(n-1)}}, (\mathscr{L}_0^{(n-1)})^{\otimes \nu k_0}\bigr)^{\SL(V^{(n-1)})}$ such that the induced section $l_{\mathscr{F}} (\sigma)\in H^0\bigl(S, (\mathscr{L}_{n-1})^{\otimes \nu k_{n-1}}\bigr)^{\SL(V)}$ fulfils $l_{\mathscr{F}}(\sigma)(s) \neq 0$. This completes the proof of Theorem~\ref{thm:semiampleness}.
\end{proof}

\section{A projective moduli space for slope-semistable sheaves} \label{sect:construction}

In this section we will carry out the construction of the ``moduli space'' of $\mu$-semistable sheaves. In Section~\ref{subsect:Langton} we shortly discuss a generalisation of Langton's Theorem to our setup, before giving the construction of the desired moduli space $M^{\mu ss}$ in Section~\ref{subsect:construction}. Afterwards, the universal properties of $M^{\mu ss}$ are established in Section~\ref{subsect:universalproperties}.

We continue to use the notation introduced in the previous section; see especially Section~\ref{setup}. 

\subsection{Compactness via Langton's Theorem}\label{subsect:Langton}
The key to proving the compactness of our yet to be constructed moduli spaces lies in the following generalisation of Langton's Theorem to the case of multipolarisations:
\begin{thm}[Langton's Theorem]\label{thm:Langton}
 Let $\mathcal{R} \supset k$ be a discrete valuation ring with field of fractions $K$, 
let $i\colon X \times \mathrm{Spec}\,K \to X \times \mathrm{Spec}\,R$ be the inclusion of the generic fibre, 
and let $j\colon X_k \to X \times \mathrm{Spec}\,R$ be the inclusion of the closed fibre in $X \times \mathrm{Spec}\,R$ over $\mathrm{Spec}\,R$. 
Then, for any $(H_1,...,H_{n-1})$-semistable torsion-free coherent sheaf $E_K$ over $X \times \mathrm{Spec}\,K$, there exists a torsion-free coherent sheaf over $X \times \mathrm{Spec}\,R$ such that $i^*E \cong E_K$ and such that $j^*E$ is torsion-free and semistable.
\end{thm}
\begin{proof}
The proof of Langton's completeness result \cite{Langton} (for slope-functions defined by a single integral ample divisor) literally works for slope with respect to multipolarisations. The key point is to note that degrees with respect to multipolarisations also can be seen as coefficients of appropriate terms in some Hilbert polynomials; cf.~\cite[p.296]{Kleiman}.
\end{proof}

By replacing Langton's original theorem with Theorem~\ref{thm:Langton}, the following result can now be obtained using the argument of \cite[Prop. 8.2.5]{HL}.
\begin{prop}\label{prop:completeimage}
Let $Z \subset S$ be any $\SL(V)$-invariant closed subvariety. If $T$ is a separated scheme of finite type over $\mathbb{C}$, and if $\varphi\colon Z \to T$ is any $\SL(V)$-invariant morphism, then the image $\varphi(Z) \subset T$ is complete. In particular, any $\SL(V)$-invariant function on $S$ is constant.
\end{prop}

\subsection{Construction of the moduli space}\label{subsect:construction}
We have seen in Theorem~\ref{thm:semiampleness} that for $\nu \in \N$ big enough, the line bundle $\mathscr{L}_{n-1}^{\otimes \nu}$ is generated by $\SL(V)$-invariant sections. Hence, it is a natural idea to construct the moduli space as an image of $S$ under the map given by invariant sections of $\mathscr{L}_{n-1}^{\otimes \nu}$ for $\nu \gg 0$.

For this, we set $W_\nu \definiere  H^0(S, \mathscr{L}_{n-1}^{\otimes \nu})^{\SL(V)}$. Since $S$ is Noetherian, for every $\nu \in \N$ such that $W_\nu$ generates $\mathscr{L}_{n-1}^{\otimes \nu}$ over $S$, there exists a finite-dimensional $\C$-vector subspace $\hat W_\nu$ of $W_\nu$ that still generates $\mathscr{L}_{n-1}^{\otimes \nu}$ over $S$. We consider the induced $\SL(V)$-invariant morphism $\varphi_{\hat{W}_\nu}\colon S \to \mathbb{P}(\hat{W}_\nu^*)$ and set $M_{\hat{W}_\nu}\definiere  \varphi_{\hat{W}_\nu}(S)$. One now considers the projective varieties $M_{\hat{W}_\nu}$ for increasing values of $\nu$. Using exactly the same arguments as in the proof of \cite[Prop.~8.2.6]{HL}, which uses the notation introduced in the previous paragraph, we obtain the following result. 
\begin{prop}[Finite generation]\label{prop:fingen}
 There exists an integer $N > 0$ such that the graded ring $\bigoplus_{k \geq 0}W_{kN}$ is generated over $W_0 = \mathbb{C}$ by finitely many elements of degree one. 
\end{prop}
We are finally in a position to define the desired moduli space.
\begin{Def}[Moduli space for slope-semistable sheaves]\label{defi:modulispace}
 Let $N \geq 1$ be a natural number with the properties spelled out in Proposition~\ref{prop:fingen} above. Then, we define the polarised variety $(M^{\mu ss}, L)$ to be the projective variety
 $$M^{\mu ss}\definiere  M^{\mu ss}(c, \Lambda) \definiere \mathrm{Proj}\Bigl( \bigoplus_{k \geq 0} H^0\bigl(S,\, \mathscr{L}_{n-1}^{\otimes kN}\bigr)^{\SL(V)}\Bigr) ,$$ together with the ample line bundle $L\definiere \mathscr{O}_{M^{\mu ss}}(1)$. Moreover, we let $\Phi\colon S \to M^{\mu ss}$ be the induced $\SL(V)$-invariant morphism with $\Phi^*(L)= \mathscr{L}_{n-1}^{\otimes N}$.
\end{Def}
\subsection{Universal properties of the moduli space}\label{subsect:universalproperties}
Although the projective variety $M^{\mu ss}$ is not a coarse moduli space in general, it nevertheless has certain universal properties, which are stated in the Main Theorem, and which we establish in the present section. 

Note that these universal properties do in fact differ from the ones stated on p.~226 of \cite{HL}, where the surface case is discussed. Note especially that it is necessary to add property (2) to the set of universal properties in order to obtain uniqueness of the resulting triple $(M^{\mu ss}, \mathscr{O}_{M^{\mu ss}(1)}, N)$.

We start by formulating and proving a result that is slightly weaker than the Main Theorem. Subsequently, we will show that by choosing the ``correct`` polarising line bundle for $M^{\mu ss}$, we obtain the universal properties stated in the Main Theorem. 

\begin{prop}\label{prop:universal_1}
 Let $\underline{M}^{\mu ss}$ denote the functor that associates to each weakly normal variety $B$ the set of isomorphism classes of $B$-flat families of $\mu$-semi\-stable sheaves of class $c$ and determinant $\Lambda$. Then, there exists a natural transformation from $\underline{M}^{\mu ss}$ to $\underline{\Hom}(\cdot , M^{\mu ss})$, mapping a family $\mathscr{E}$ to a classifying morphism $\Phi_\mathscr{E}$, with the following properties:
\begin{enumerate}
 \item For every $B$-flat family $\mathscr{E}$ of $\mu$-semistable sheaves of class $c$ and determinant $\Lambda$ with induced classifying morphism $\Phi_{\mathscr{E}}\colon B \to M^{\mu ss}$ we have
\begin{equation}\label{eq:functoriallinebundlepullback}
 \Phi_{\mathscr{E}}^* (L) \cong \lambda_{\mathscr{E}}\bigl( {u}_{n-1}(c)\bigr)^{\otimes N},
\end{equation}
where $\lambda_{\mathscr{E}}\bigl( {u}_{n-1}(c)\bigr)$ is the determinant line bundle on $B$ induced by $\mathscr{E}$ and $ {u}_{n-1} (c)$; cf.~Section~\ref{subsect:Grothendieckgroup}. 
\item For any other triple of a natural number $N'$, a projective variety $M'$, and an ample line bundle $L'$ fulfilling the conditions spelled out in (1), there exist a natural number $d \in \mathbb{N}^{>0}$, and a uniquely determined morphism $\psi\colon M^{\mu ss} \to M'$ such that $\psi^*(L')^{\otimes dN} \cong L^{\otimes dN'}$. 
\end{enumerate}
\end{prop}
In the subsequent proofs we will use the following standard terminology.
\begin{Def}
 Let $X$ be a proper variety, and $L$ a line bundle on $X$. Then, the \emph{section ring} of $L$ is defined to be
\[R(X, L) \definiere  \bigoplus_{k \geq 0}H^0\bigl(X, L^{\otimes k} \bigr).\]
For any $d \geq 2$, we define the \emph{$d$-th Veronese subring} of $R(X,L)$ to be
\[R(X,L)_{(d)} \definiere  \bigoplus_{d|k} H^0\bigl(X, L^{\otimes k} \bigr) \subset R(X,L).\]
\end{Def}
\begin{proof}[Proof of Proposition~\ref{prop:universal_1}]
In the proof of part (1) we follow \cite[proof of Lem.~4.3.1]{HL}. We will use the notation introduced in Section \ref{setup}.

 Let $B$ be a weakly normal variety and $\mathscr{E}$ a $B$-flat family of $\mu$-semistable sheaves with Hilbert polynomial $P$ with respect to the chosen ample polarisation $\mathscr{O}_X(1)$. 
 Denote the natural projections of $B \times X$ by $p\colon B \times X \to B$ and $q\colon B \times X \to X$. 

 Let $m \in \mathbb{N}$ be as in Section \ref{setup}, such that every 
$\mu$-semistable sheaf $F$ with the given invariants is $m$-regular. 
  We set $V \definiere  \mathbb{C}^{\oplus P(m)}$ and $\mathscr{H} \definiere  V \otimes \mathscr{O}_X (-m)$. The sheaf
$V_{\mathscr{E}} \definiere  p_*\bigl(\mathscr{E} \otimes q^*\mathscr{O}_X(m)\bigr)$
is locally free of rank $P(m)$, and there is a canonical surjection $\varphi_\mathscr{E}\colon p^* V_\mathscr{E} \otimes q^*\mathscr{O}_X(-m) \twoheadrightarrow \mathscr{E}$. 

Let $\pi\colon R(\mathscr{E}) \to B$ be the frame bundle 
associated with $V_{\mathscr{E}}$; the group $\GL(V)$ acts on $R(\mathscr{E})$ making it a $\GL(V)$-principal bundle with good quotient $\pi$. Since the pullback of $V_\mathscr{E}$ to $R(\mathscr{E})$ has a universal trivialisation, we obtain a canonically defined quotient $\widetilde q _{\mathscr{E}}\colon \mathscr{O}_{R(\mathscr{E})} \otimes_\mathbb{C} \mathscr{H} \twoheadrightarrow (\pi \times \mathrm{id}_X)^*\mathscr{E} $ on $R (\mathscr{E}) \times X$. The quotient $\widetilde q _{\mathscr{E}}$ gives rise to a classifying morphism $\widetilde \Psi_\mathscr{E}\colon R(\mathscr{E}) \to \mathrm{Quot}(\mathscr{H}, P)$, which is equivariant with respect to the $\GL(V)$-actions on $R(\mathscr{E})$ and $R^{\mu ss}$. Here, the latter action is induced by the $\SL(V)$-action via $\PGL(V)$, cf.~\cite[Lem.~4.3.2]{HL}.

Since the sheaf $\mathscr{E}_b$ was assumed to be slope-semistable for all $b \in B$, the image of $\widetilde \Psi_\mathscr{E}$ is contained in $R^{\mu ss}$. Note that as a principal bundle over the seminormal variety $B$, the space $R(\mathscr{E})$ is itself seminormal. Consequently, by Proposition~\ref{propnot:weaknormalisation} the map $\widetilde \Psi_\mathscr{E}$ lifts to a morphism from $R(\mathscr{E})$ to $(R^{\mu ss})^{wn} =S$, which we will continue to denote by $\widetilde \Psi_\mathscr{E}$. Composing $\widetilde \Psi_\mathscr{E}$ with the $\SL(V)$-invariant morphism $\Phi\colon S \to M^{\mu ss}$, we obtain a $\GL(V)$-invariant morphism $\widetilde \Phi_\mathscr{E}\colon R(\mathscr{E}) \to M^{\mu ss}$. Since $\pi$ is a good quotient, and hence in particular a categorical quotient, there exists a uniquely determined morphism $\Phi_\mathscr{E}\colon B \to M^{\mu ss}$ such that the following diagram commutes:
\begin{equation}\label{eq:framebundlediagram}
\begin{gathered}
\begin{xymatrix}{
 R(\mathscr{E}) \ar[rd]^{\widetilde\Phi_{\mathscr{E}}} \ar[d]_\pi \ar[r]^{\widetilde \Psi_\mathscr{E}}&  S \ar[d]^{\Phi} \\
 B \ar[r]^<<<<<{\Phi_\mathscr{E}}&  M^{\mu ss}.
}
\end{xymatrix}
\end{gathered}
\end{equation}
Assigning $\Phi_\mathscr{E} \in \mathrm{Mor}(B, M^{\mu ss})$ to $\mathscr{E}$ yields the desired natural transformation $\underline{M}^{\mu ss} \to \underline{\Hom}(\cdot , M^{\mu ss})$. 

It remains to show the isomorphism of line bundles \eqref{eq:functoriallinebundlepullback}. It follows from the commutative diagram \eqref{eq:framebundlediagram} and from the definition of $(M^{\mu ss}, \mathscr{O}_{M^{\mu ss}}, N)$ that
\[\pi^* \Phi_{\mathscr{E}}^* \mathscr{O}_{M^{\mu ss}}(1) \cong \widetilde \Psi_{\mathscr{E}}^* (\mathscr{L}_{n-1}^{\otimes N}) \cong\lambda_{\pi^*\mathscr{E} }(  u_{n-1})^{\otimes N} \cong \pi^*\lambda(  u_{n-1})^{\otimes N}.\]
Since $\pi^*\colon \mathrm{Pic}(B) \to \mathrm{Pic}(R(\mathscr{E}))$ is injective by \cite[Ch.~1, \S3, Prop.~1.4]{MumfordGIT}, cf.~\cite[Lem.~2.14]{LePotierDonaldsonUhlenbeck}, the previous chain of isomorphisms implies
\[\Phi_{\mathscr{E}}^* \mathscr{O}_{M^{\mu ss}}(1) \cong \lambda(  u_{n-1})^{\otimes N}, \]
as claimed in (1).

Next, we prove the claims made in item (2). If $(M', L', N') $ is another triple satisfying the conditions of (1), then applying the universal property we obtain a uniquely determined morphism $\Phi'\colon S \to M'$ such that $(\Phi')^*(L') \cong \mathscr{L}_{n-1}^{\otimes N'}$. We claim that $\Phi'$ is $\SL(V)$-invariant. Indeed, the restriction of $\Phi'$ to an $\SL(V)$-orbit $\mathcal{O}$ gives the classifying map for the restriction of the universal family to $\mathcal{O}$. But this latter classifying map is constant, which implies that $\Phi'(\mathcal{O}) = \{\mathrm{pt.}\}$, as claimed. 

Let $d \in \mathbb{N}^{>0}$ such that the $d$-th Veronese subring $R(M', L')_{(d)}$ of $R(M',L')$ is generated in degree one. Let $f\colon M' \to \mathrm{Proj}(R(M', L')_{(d)})$ be the natural isomorphism. Then, the pullback $f^*(\mathscr{O}(1))$ is naturally isomorphic to $(L')^{\otimes d}$. Setting $d '\definiere  dN'$ and using the universal property of $L'$, we obtain a natural morphism of graded rings 
\[ R(M', L')_{(dN)} \to  \bigoplus_{ d' | k } H^0(S, \mathscr{L}_{n-1}^{\otimes kN})^{\SL(V)}  \] by pulling back sections via the $\SL(V)$-invariant map $\Phi'$. This in turn induces a uniquely determined morphism  \[\psi\colon M = \mathrm{Proj}\Big(\bigoplus_{ d' | k } H^0(S, \mathscr{L}_{n-1}^{\otimes kN})^{\SL(V)} \Big) \to \mathrm{Proj}\Big(R(M', L')_{(dN)} \Big) = M'\] such that $\psi^*(L')^{\otimes dN} = L^{\otimes dN'}$. This concludes the proof of (2).
\end{proof}

\begin{cor}[Weak normality]
 The variety $M^{\mu ss}$ is weakly normal. 
\end{cor}
\begin{proof}
 The weak normalisation $\eta\colon~(M^{\mu ss})^{wn}\to M^{\mu ss}$ together with the line bundle $L' \definiere  \eta^* L$ and the natural number $N$ has the universal properties spelled out in part (1) of Proposition~\ref{prop:universal_1}. Indeed, by Proposition~\ref{propnot:weaknormalisation} every map from a weakly normal variety $S$ to $M^{\mu ss}$ can be lifted in a unique way to a map from $S$ to $(M^{\mu ss})^{wn}$ satisfying the required pullback properties. Consequently, part (2) of Proposition~\ref{prop:universal_1} yields a uniquely determined morphism $\psi\colon M^{\mu ss} \to (M^{\mu ss})^{wn}$, which gives an inverse to $\eta$. 
\end{proof}
Comparing the universal property $(2)$ of Proposition~\ref{prop:universal_1} with the one claimed in the Main Theorem, we see that we need to improve the uniqueness statement in Proposition~\ref{prop:universal_1}(2). This amounts to showing that although the map $\Phi\colon S \to M^{\mu ss}$ is not proper, due to its $\SL(V)$-invariance it still has certain properties that are reminiscent of the Stein fibration and Iitaka fibration for semiample line bundles on proper varieties, cf.~\cite[Sect.~2.1.C]{Lazarsfeld}.

\begin{Lem}[Pushing-down invariant functions]\label{lem:categorical}
 Let $\Phi\colon S \to M^{\mu ss}$ be as before. Then, we have
\[ \Phi_*(\mathscr{O}_S)^{\SL(V)} = \mathscr{O}_{M^{\mu ss}}.\]
\end{Lem}
\begin{proof}
 We first note the following generalisation of the classical projection formula to the equivariant setting, cf.~\cite[Lem.~9.2]{PaHq}:
 
\begin{Lem}[Equivariant projection formula]\label{lem:equivariantprojection}
Let $G$ be an algebraic group, let $Y$ be an algebraic $G$-variety, and let $f\colon Y \to Z$ be a $G$-invariant morphism to an algebraic variety Z. Then, for every $G$-linearised coherent algebraic sheaf $\mathscr{F}$ on $Y$ and every locally free sheaf $\mathscr{E}$ of finite rank on $Z$, there is a natural isomorphism
\[ f_* (\mathscr{F} \otimes f^*\mathscr{E})^G \cong f_*(\mathscr{F})^G \otimes \mathscr{E}. \] 
Here, $f^*\mathscr{E}$ is given the natural $G$-linearisation as a pullback bundle via an invariant morphism, and $\mathscr{F}\otimes f^*\mathscr{E}$ is given the natural tensor product linearisation. 
\end{Lem}
We proceed as follows. Since $\Phi$ is invariant, we obtain a natural morphism $\theta\colon \mathscr{O}_{M^{\mu ss}} \to \Phi_*(\mathscr{O}_S)^{\SL(V)}$ of quasi-coherent sheaves of $\mathscr{O}_{M^{\mu ss}}$-modules. Since $L = \mathscr{O}_{M^{\mu ss}}(1)$ is ample, and since $R(M^{\mu ss}, L)$ is generated in degree one, in order for $\theta$ to be an isomorphism it suffices to show that the induced map \[\hat \theta_k\colon H^0\bigl(M^{\mu ss}, L^{\otimes k}\bigr) \to 
H^0\bigl(M^{\mu ss}, \Phi_*(\mathscr{O}_S)^{\SL(V)} \otimes L^{\otimes k}\bigr)\] is an isomorphism for all natural numbers $k \geq 1$. For this, we note that $\hat \theta_k$ can be factored in the following way:
\begin{align*}
 H^0\bigl(M^{\mu ss}, L^{\otimes k}\bigr)  &\overset{\alpha_k}{\longrightarrow } H^0\bigl(S, \mathscr{L}_{n-1}^{\otimes kN}\bigr)^{\SL(V)} \overset{\beta_k}{\longrightarrow} H^0\bigl(M^{\mu ss}, \Phi_*(\mathscr{O}_S \otimes \Phi^*L^{\otimes k})^{\SL(V)}\bigr) \\
 &\overset{\gamma_k}{\longrightarrow} H^0\bigl(M^{\mu ss}, \Phi_*(\mathscr{O})^{\SL(V)} \otimes L^{\otimes k}\bigr).
\end{align*}
In the previous diagram, $\alpha_k$ and $\beta_k$ are isomorphisms by definition and by the $\SL(V)$-equivariant isomorphism $\Phi^* L^{ \otimes k} \cong \mathscr{L}_{n-1}^{\otimes kN}$, and $\gamma_k$ is an isomorphism by the equivariant projection formula, Lemma~\ref{lem:equivariantprojection} above. Consequently, the composition, which is equal to $\hat \theta_k$, is an isomorphism. This concludes the proof of Lemma~\ref{lem:categorical}.\qedhere
\end{proof}
The following is the analogue of \cite[Ex.~2.1.14]{Lazarsfeld} in our equivariant setup. 
\begin{lemma}[Injectivity of equivariant pullback] \label{lem:equivpullbackinjective}
The natural pullback map from $\mathrm{Pic}(M^{\mu ss})$ to the group of $\SL(V)$-linearised line bundles on $S$, 
\[\Phi^*\colon \mathrm{Pic}(M^{\mu ss}) \to \mathrm{Pic}_{\SL(V)}(S),\]
is injective. 
\end{lemma}
\begin{proof}
 Let $D$ be a line bundle on $M^{\mu ss}$ such that $\Phi^*(D)$ is $\SL(V)$-equivariantly isomorphic to the trivial line bundle with the trivial $\SL(V)$-linearisation. The equivariant projection formula, Lemma~\ref{lem:equivariantprojection}, and Lemma~\ref{lem:categorical} lead to the following chain of isomorphisms
\begin{equation}\label{eq:invariantchain}
\begin{aligned}
H^0\bigl(S, \mathscr{O}_S \bigr)^{\SL(V)} &\cong H^0\bigl(M^{\mu ss}, \Phi_*(\Phi^*D)^{\SL(V)} \bigr) \\
					      &\cong  H^0\bigl(M^{\mu ss}, \Phi_*(\mathscr{O}_S)^{\SL(V)} \otimes D \bigr) \\
					      &\cong H^0\bigl(M^{\mu ss}, D \bigr).
\end{aligned}
\end{equation}
Any non-zero constant function on $S$ is trivially $\SL(V)$-invariant, and hence via \eqref{eq:invariantchain} induces a section in $H^0\bigl(M^{\mu ss}, D \bigr)$ that does not vanish on any component of $M^{\mu ss}$. As the same reasoning applies to the dual $D^{-1}$, the line bundle $D$ is trivial, as claimed.
\end{proof}
\begin{lemma}[Injectivity of pullback]\label{lem:pullbackinjective}
 The pullback map $\Phi^*\colon \mathrm{Pic}(M^{\mu ss}) \to \mathrm{Pic}(S)$
is injective.
\end{lemma}
\begin{proof}
 Since $\SL(V)$ is semisimple, and therefore has no nontrivial characters, by \cite[Chap.~1, \S3, Prop.~1.4]{MumfordGIT} the forgetful map $\mathrm{Pic}_{\SL(V)}(S) \to \mathrm{Pic}(S)$ is injective. Together with Lemma~\ref{lem:equivpullbackinjective} this yields the claim.
\end{proof}

In the next step, we make a first improvement concerning the universal properties of Proposition~\ref{prop:universal_1}.

\begin{lemma}\label{lem:universal_2}
 Using the notation of Proposition~\ref{prop:universal_1}, we may take $d=1$. I.e., for any other triple of a natural number $N'$, a projective variety $M'$, and an ample line bundle $L'$ fulfilling the conditions spelled out in part (1) of Proposition~\ref{prop:universal_1}, there exist a uniquely determined morphism $\psi\colon M^{\mu ss} \to M'$ such that $\psi^*(L')^{\otimes N} \cong L^{\otimes N'}$. Call this universal property \emph{Property (2')}.
\end{lemma}

\begin{proof}
Let $(M', L', N')$ be a second triple having property (1) of Proposition~\ref{prop:universal_1}. Then, as in the proof of Proposition~\ref{prop:universal_1} let $\Phi'\colon S \to M'$ be the classifying morphism for the universal family over $S$. Via the Proj construction, we obtain a uniquely determined morphism $\psi\colon M^{\mu ss} \to M'$ such that the following diagram commutes
\[\begin{xymatrix}{
    & S \ar[ld]_{\Phi} \ar[rd]^{\Phi'}&    \\
M^{\mu ss} \ar[rr]^{\psi} &  &  M'.
}
  \end{xymatrix} 
\]
By the universal properties of $L$ and $L'$, we have natural isomorphisms
\[\Phi^*(\psi^*(L')^{\otimes N}) \cong (\Phi')^*(L')^{\otimes N} \cong \mathscr{L}_{n-1}^{\otimes NN'} \cong \Phi^*(L^{\otimes N'}). \]
As $\Phi^{*}$ is injective by Lemma~\ref{lem:pullbackinjective}, this implies $\psi^*(L')^{\otimes N} \cong L^{\otimes N'}$.
\end{proof}

We are now in a position to prove the existence of an ``optimal'' line bundle on $M^{\mu ss}$ that has the universal property stated in the Main Theorem. 

\begin{prop}\label{prop:optimaluniversal} 
Let $N$ be minimal such that $( M^{\mu ss}, L, N)$ has
property (1) of Propostion~\ref{prop:universal_1}
and Property (2'), as stated in Lemma~\ref{lem:universal_2}. 
Then, $(M \cong M^{\mu ss}, L, N)$ has the following universal property (2''): for any other triple of a natural number $N'$, a projective variety $M'$, and an ample line bundle $L'$ fulfilling the conditions spelled out in part (1) of Proposition~\ref{prop:universal_1}, 
we have $N|N'$, and 
there exist a uniquely determined morphism $\psi\colon M^{\mu ss} \to M'$ 
such that 
$\psi^*(L') \cong L^{\otimes(N'/N)}$.
\end{prop}

\begin{proof}
Consider a second triple of a natural number $N'$, a projective variety $M'$, and an ample line bundle $L'$ that fulfils the conditions spelled out in item (1) of Proposition~\ref{prop:universal_1}, and let $\psi\colon M \to M'$ be the uniquely determined morphism of Lemma~\ref{lem:universal_2}. To establish the claim, it suffices to show that $N|N'$, as the rest follows by the same argument as in the proof of Lemma~\ref{lem:universal_2}. 

Suppose that $N$ does not divide $N'$ and let
\begin{equation}\label{eq:lcd}
    e=\mathrm{lcd}(N, N') < N                                         
\end{equation}
 be their largest common divisor. There exist $a, b \in \mathbb{Z}$ such that $e = aN' + bN$, and we set $A \definiere  \psi^*(L')^{\otimes a} \otimes L^{\otimes b}$. The pullback of $A$ to $S$ via $\Phi$ equals $\mathscr{L}_{n-1}^{\otimes (aN'+bN)}=\mathscr{L}_{n-1}^{\otimes e}$. From this we infer that $A^{\otimes(N/e)} \cong L$, using the injectivity of $\Phi^*$ provided by Lemma \ref{lem:pullbackinjective}. Hence, $A$ is ample. Aiming for a contradiction to the minimality of $N$, we claim that the triple $(M, A, e)$ has the universal properties (1) and (2'). 

As to property (1), let $B$ be a weakly normal variety, and  let $\mathscr{E}$ be a $B$-flat family of $\mu$-semistable sheaves of class $c$ and determinant $\Lambda$  on $X$. The family $\mathscr{E}$ induces two classifying morphisms $\Phi_{\mathscr{E}}\colon B \to M=M^{\mu ss}$ and $\Phi'_{\mathscr{E}}\colon B \to M'$. The morphism $\psi\colon M\to M'$ given by property (2) for $M$ is such that $\psi\circ \Phi_{\mathscr{E}}= \Phi'_{\mathscr{E}}$, as can be seen from the proof of Proposition~\ref{prop:universal_1}. We conclude that
\begin{equation}\label{eq:twosclassifyingmorphisms}
 \Phi_{\mathscr{E}}^*(A)=(\Phi'_{\mathscr{E}})^*((L')^{\otimes a} )\otimes\Phi_{\mathscr{E}}^*( L^{\otimes b})\cong \lambda_{\mathscr{E}}( {u}_{n-1})^{aN'+bN}= \lambda_{\mathscr{E}}( {u}_{n-1})^e.
\end{equation}
 
 In order to prove that $(M, A, e)$ enjoys property (2'), let $(M'',L'', N'')$ be any other triple having property (1). Then, as in the proof of Lemma~\ref{lem:universal_2} there exists a natural commutative diagram 
\[\begin{xymatrix}{
    & S \ar[ld]_{\Phi} \ar[rd]^{\Phi''}&    \\
M^{\mu ss} \ar[rr]^{\psi} &  &  M''.
}
  \end{xymatrix} 
\]
From the construction and from \eqref{eq:twosclassifyingmorphisms}, we obtain natural isomorphisms
$$\Phi^*(A^{\otimes N''})\cong \mathscr{L}_{n-1}^{\otimes eN''} \cong (\Phi'')^*(L'')^{\otimes e}\cong \Phi^*(\psi^*(L'')^{\otimes e})$$ 
on $S$. Using Lemma~\ref{lem:pullbackinjective} another time, we conclude that $A^{\otimes N''} \cong \psi^*(L''^{\otimes e})$. 

In summary, we have established property (1) and (2') for the new triple $(M, A, e)$. Together with the inequality $\eqref{eq:lcd}$, this yields a contradiction to the minimality of $N$. We therefore conclude that $N|N'$, which is exactly the claim made in Proposition~\ref{prop:optimaluniversal}.
\end{proof}

As an immediate consequence, we obtain in the usual way:
\begin{cor}[Uniqueness]
 The triple $(M^{\mu ss}, L, N)$ with $N$ minimal as above is uniquely determined up to isomorphism by the properties $(1)$ and $(2'')$. 
\end{cor}
\begin{rem}With Proposition~\ref{prop:optimaluniversal} at hand, we have now established all claims made in the Main Theorem.
\end{rem}

\section{Geometry of the moduli space}\label{sect:geometry}
In the current section we start to investigate the geometry of $M^{\mu ss}$. First, in Section~\ref{subsect:separation} we look at the separation properties of the map $\Phi\colon S \to M^{\mu ss}$, second, in Section~\ref{subsect:compactificationofsimpl} we prove that $M^{\mu ss}$ provides a compactification for the moduli space of $\mu$-stable reflexive sheaves, and third, in Section~\ref{slopevsGM} we investigate the relation between $M^{\mu ss}$ and the Gieseker-Maruyama moduli space in the special case where $H_1 = \ldots = H_{n-1}$. 
\subsection{Separation properties} \label{subsect:separation}The geometry of the map $\Phi$ will be studied in terms of Jordan-H\"older filtrations. Let us introduce the relevant terminology.
\begin{propnot}[Jordan-H\"older filtrations]
Let $X$ be a projective $n$-dimensional manifold, and let $(H_1,\ldots,H_{n-1})$ be a multipolarisation on $X$ with respect to which we consider slope-semistability.
 For a $\mu$-semistable sheaf $F$ on $X$, there exist $\mu$-Jordan-H\"older filtrations (in the sense of \cite[Def.~1.5.1]{HL}). 
Let $gr^{\mu}F$ denote the graded sheaf associated with a $\mu$-Jordan-H\"older filtration with torsion-free factors and set $F^{\sharp}\definiere (gr^{\mu}F)^{\vee\vee}$. Then, $F^{\sharp}$ is a reflexive $\mu$-polystable sheaf on $X$, which depends only on $F$ and not on the chosen Jordan-H\"older filtration.
\end{propnot}
\begin{proof}
 Existence of $\mu$-Jordan-H\"older filtrations is shown exactly as in \cite[Prop.~1.5.2]{HL}, uniqueness as in \cite[Cor.~1.6.10]{HL}. See also \cite{Miyaoka}. 
\end{proof}

\begin{Not}
For any $\mu$-semistable sheaf $F$ as above we consider the natural map $\imath\colon gr^{\mu}F \to F^\sharp$. 
Since $gr^{\mu}F$ is torsion-free, $\imath$ is injective, and the quotient sheaf $F^{\sharp}/gr^{\mu}F$ is supported in codimension at least two. 
We associate to $F$ the two-codimensional support Chow-cycle of the sheaf  $F^{\sharp}/gr^{\mu}F$, which we will denote by $C_F$. 
\end{Not}
\begin{rems}
a) For later reference, we quickly recall how the cycle $C_F$ can be computed; 
for this we set $T_F\definiere F^{\sharp}/gr^{\mu}F$, and let $C_k$, $k=1, \dots, K$, be the codimension-two components of $\mathrm{supp}(T_F)$. 
Let $X^{(n-2)}$ be a general complete intersection surface of $X$ (with respect to some multipolarisation $(H_1,\ldots,H_{n-2})$ consisting of very ample line bundles), 
and let $p^{(k)}_1, \dots, p^{(k)}_{N_k}$ be the components of the scheme-theoretic intersection $X^{(n-2)} \cap C_k$, which in our situation will be smooth. 
The restriction $T_F|_{X^{(n-2)}}$ is a skyscraper sheaf. The associated natural number \[\mathrm{length}_{\mathscr{O}_{X^{(n-2)}, p_i^{(k)}}} \bigl(T_F|_{X^{(n-2)}}\bigr)_{p_i^{(k)}}\] is independent of $i \in \{1, \dots, N_k\}$, and will be called $m_k$. With these notations at hand, we have \[C_F = \sum_{k= 1}^K m_k C_k \in \mathrm{Chow}_{n-2}(X).\] Moreover, for further reference, in the situation under consideration we let 
\[C_F^{X^{(n-2)}} \definiere  \sum_{k=1}^{K}\sum_{i=1}^{N_k} m_k[p_i^{(k)}] \in \mathrm{Sym}^{*}\bigl(X^{(n-2)}\bigr)\]
 be the $0$-cycle on $X^{(n-2)}$ defined by $T_F|_{X^{(n-2)}}$.

b) The cycle $C_F$ depends only on $F$ and not on the chosen Jordan-H\"older filtration, as can be seen by a reduction to the surface case \cite[Cor.~1.5.10]{HL} using hyperplane sections, and by the description of $C_F$ given in a).
\end{rems}

The connection between Jordan-H\"older filtrations and restriction of $\mu$-semi\-sta\-ble sheaves to curves is established by the following result, variants of which appear throughout the literature.
\begin{Lem}\label{lem:separation}
Let $F_1$, $F_2$ be $\mu$-semistable torsion-free sheaves on $X$. If $m_1$, $\ldots$, $m_{n-1}$ are sufficiently large integers, then $F_1^{\sharp}$ and $F_2^{\sharp}$ are isomorphic if and only if 
the restrictions of $F_1$ and $F_2$ to any general complete intersection curve $X^{(n-1)}\definiere X_1\cap\ldots \cap X_{n-1}$, $X_j\in |m_j H_j|$, are S-equivalent.  
\end{Lem}

\begin{proof}
Choose Jordan-H\"older filtrations with torsion-free factors for $F_1$ and $F_2$. If $m_1$, $\ldots$, $m_{n-1}$ are sufficiently large integers, then a general complete intersection curve $X^{(n-1)}=X_1\cap\ldots \cap X_{n-1}$ will have the following properties:
\begin{itemize}
 \item[(i)] the curve $X^{(n-1)}$ avoids the singularities of  $gr^{\mu}F_1$ and $gr^{\mu}F_2$, and 
 \item[(ii)] the restriction of the 
Jordan-H\"older filtrations of  $F_1$ and $F_2$ to $X^{(n-1)}$ are  Jordan-H\"older filtrations for  $F_1|_{X^{(n-1)}}$ and $F_2|_{X^{(n-1)}}$.
\end{itemize}
Item (i) is achievable, since torsion-free sheaves are locally free in codimension one 
 i.e., their singularities lie in codimension two or higher. Item (ii) is achievable by the Semistable Restriction Theorem, Proposition~\ref{prop:slopeMR}.

Then, by item~(ii) the restricted sheaves $F_1|_{X^{(n-1)}}$ and $F_2|_{X^{(n-1)}}$  are S-equivalent if and only if $(gr^{\mu}F_1)|_{X^{(n-1)}}\cong (gr^{\mu}F_2)|_{X^{(n-1)}}$, and this in turn is equivalent to $F_1^{\sharp}|_{X^{(n-1)}}\cong F_2^{\sharp}|_{X^{(n-1)}}$ by item (i).

 In the next step, we focus on the reflexive sheaf  $E\definiere \mathscr{H}\!om(F_1^{\sharp},F_2^{\sharp})$. 
If $E$ is any reflexive sheaf on a smooth projective manifold $Z$, then its restriction to a general (smooth) hyperplane section stays reflexive, see \cite[Lem.~1.1.12 and Cor.~1.1.14]{HL}. Moreover, $H^1\bigl(Z, E(-mH)\bigr)$ vanishes when $H$ is ample and $m\gg 0$ by \cite[Ch.~IV, Thm.~3.1]{BSt_engl}. Therefore, if $m_i\gg 0$ and $X_i\in|m_iH_i|$, $1\le i\le n-1$, are general elements, it follows that
\begin{enumerate}
 \item[($\alpha$)] each $X^{(i)}\definiere X_1\cap\ldots \cap X_{i}$ is smooth, 
 \item[($\beta$)] each $E|_{X^{(i)}}$ is reflexive on $X^{(i)}$,
 \item[($\gamma$)] the following cohomology groups vanish: \[H^1\bigl(X, E(-X_1)\bigr)= H^1\bigl(X^{(i-1)},E|_{X^{(i-1)}}(-X_{i})\bigr) = \{0\} \quad \text{ for } 2 \leq i \leq n-1.\]
\end{enumerate}
In this setup, we can lift sections from $H^0(X^{(n-1)}, E|_{X^{(n-1)}})$ to $H^0(X,E)$ using the vanishing in item ($\gamma$) and the exact sequences
\begin{align*}
0\to E|_{X^{(i)}}(-X_{i+1})\to E|_{X^{(i)}}\to  E|_{X^{(i+1)}}\to0.
\end{align*}
It follows that $F_1^{\sharp}|_{X^{(n-1)}}\cong F_2^{\sharp}|_{X^{(n-1)}}$ if and only if $F_1^{\sharp}$ and $F_2^{\sharp}$ are isomorphic on $X$, cf. \cite[Proposition IV.1.7 (2)]{KobayashiVectorBundles}. This completes the proof.
\end{proof} 

We now formulate a first separation criterion, which describes the geometry of $\Phi$ and hence of $M^{\mu ss}$.

\begin{thm}[Separating semistable sheaves in the moduli space]\label{thm:sep}
 Let $F_1$ and $F_2$ be two $(H_1,\ldots,H_{n-1})$-semistable sheaves on the projective manifold $X$ such that  $F_1^{\sharp} \not \cong F_2^{\sharp}$ or $C_{F_1}\neq C_{F_2}$. 
Then, 
$F_1$, $F_2$ give rise to distinct points in   
 $ M^{\mu ss}$.
\end{thm}
\begin{proof} 
We will use the setup and notation introduced in Section~\ref{setup}. 
We look for invariant sections of $\mathscr{L}_{n-1}^{\otimes \nu}$ on $S=(R^{\mu ss}_{red})^{wn}$ 
that separate the orbits corresponding to $F_1$ and $F_2$. 
For this we follow the proof of Theorem~\ref{thm:semiampleness} and 
restrict the $S$-flat family $\mathscr{F}$  successively to appropriately 
chosen general hyperplane sections $X^{(1)}, \ldots, X^{(n-1)}$. 

\subsubsection*{Simplifying assumptions} 
As said before, we work in the setup and notation introduced in Section~\ref{setup}. 
Similarly to the argument just after Remark~\ref{rem:semistabilityoncurve}, 
it follows from Lemma~\ref{flatness2} (applied to the very ample line bundles $a_iH_i$) and Lemma~\ref{sections} 
that in order to prove the existence of sections in powers of $\mathscr{L}_{n-1}$ that separate $F_1$ and $F_2$, 
we may assume without loss of generality to be in the following setup:
\begin{Setup}For any complete intersection curve $X^{(n-1)}$ 
obtained by intersecting general members $X_i$ of $|a_iH_i|$, $1\le i\le n-1$,
the following holds: if we set $\mathscr{F}^{(i)} \definiere  \mathscr{F}|_{S \times X^{(i)}}$, then
\begin{enumerate}
 \item[(i)] all families $\mathscr{F}^{(i)}$ are $S$-flat, 
 \item[(ii)] the sequences of equivariant $S$-flat sheaves
 \begin{align}
  0  \to \mathscr{F}^{(i-1)}(- X^{(i)}) \to \mathscr{F}^{(i-1)} \to \mathscr{F}^{(i)} \to 0
 \end{align}
are exact.
\end{enumerate} 
\end{Setup}
Our sheaves $F_1$, $F_2$ are represented by points $s_1, s_2\in S$. An application of \cite[Lem.~1.1.12 and Cor.~1.1.14]{HL} allows us to assume that their restrictions to $X^{(i)}$ remain torsion-free for all $i \in \{1, \ldots, n-2\}$.

After these preparatory considerations, the proof of Theorem~\ref{thm:sep} will be divided into two cases.
\subsubsection*{Case 1: Separating sheaves with different $F^\sharp$:} 
If $F_1^{\sharp} \not \cong F_2^{\sharp}$, then by Lemma \ref{lem:separation} and by the Semistable Restriction Theorem, Proposition~\ref{prop:slopeMR}, their restrictions to a general complete intersection curve $X^{(n-1)}$ are locally free, semistable, and not $S$-equivalent. Consequently, it follows from  Lemma~\ref{lem:GIT=semistable} 
that the corresponding points in the Quot scheme $Q_{X^{(n-1)}}$ may be separated by an invariant section $\sigma$ 
in a sufficiently high tensor power $(\mathscr{L}^{(n-1)}_0)^{\otimes \nu}$ of $\mathscr{L}^{(n-1)}_0$. 
The associated lift $l_{\mathscr{F}}^{X^{(n-1)}}(\sigma) \in H^0\bigl(S, (\mathscr{L}_{n-1}\bigr)^{\otimes \nu k})^{\SL(V)}$ 
separates the points $s_1 \in S$ and $s_2 \in S$ corresponding to $F_1$ and $F_2$, respectively; cf.~the proof of Theorem~\ref{thm:semiampleness}. Consequently, $F_1$ and $F_2$ give rise to different points in $M^{\mu ss}$, which is the desired conclusion.

\subsubsection*{Case 2: Separating sheaves with identical $F^\sharp$:} 
Case 1 being already established, we may assume that $F_1$ and $F_2$ are two $\mu$-semistable sheaves on $X$ such that 
\begin{enumerate}
 \item[($\alpha$)] $F_1^\sharp \cong F_2^\sharp$, but
 \item[($\beta$)] $C_{F_1} \neq C_{F_2}$.
\end{enumerate}

Working towards our goal of separating $F_1$ and $F_2$ in $M^{\mu ss}$, we first reduce to the case of polystable sheaves, cf.~the proof of \cite[Thm.~8.2.11]{HL}: If $F$ is $\mu$-semistable, and if $gr^\mu(F)$ is the torsion-free graded sheaf associated with a $\mu$-Jordan-H\"older filtration of $F$, then there exists a flat family $\mathscr{G}$ parametrised by $\mathbb{C}$ such that $\mathscr{G}_0 \cong gr^\mu(F)$ and $\mathscr{G}_t \cong F$ for all $t \neq 0$. As $M^{\mu ss}$ is separated, $F$ and $gr^\mu(F)$ are mapped to the same point $p \in M^{\mu ss}$.

As a consequence of the previous paragraph, in the following we may assume without loss of generality that both $F_1$ and $F_2$ are $\mu$-polystable, with the same double dual $F_1^\sharp \cong F_2^{\sharp} =: E$. Hence, there exist two exact sequences
\[0 \to F_j \to E \to T_j \to 0 \quad \quad j=1,2.\] We let $s_1, s_2 \in S$ be two points with $\mathscr{F}_{s_i} \cong F_i$, $i = 1,2$. 

The idea of our proof is as follows: As the cycles associated with the two sheaves $F_1$ and $F_s$ differ, their respective intersection with any general complete intersection \emph{surface} will still differ. Close reading of the proof of the separation properties of determinant line bundles for families of sheaves on surfaces will then show that there are sections in determinant line bundles associated with families over complete intersection \emph{curves} that will induce separating sections in $\mathscr{L}_{n-1}$, as desired.

By the construction of $C_{F_i}$ and by the Semistable Restriction Theorem, Proposition~\ref{prop:slopeMR}, for any general complete intersection surface $X^{(n-2)}$ there exists an open subset $U_2$ in the linear system $\bigl|m_{n-1} H_{n-1}|_{X^{(n-2)}}\bigr|$ such that 
 the following statements hold for $X^{(n-2)}$ and for any curve $X^{(n-1)} \in U_2$:
\begin{enumerate}
 \item[(i)] The cycles $C_{F_1}^{X^{(n-2)}}$ and $C_{F_2}^{X^{(n-2)}}$ in $\mathrm{Sym}^*(X^{(n-2)})$ do not coincide. 
 \item[(ii)] Both $F_1|_{X^{(n-2)}}$ and $F_2|_{X^{(n-2)}}$ are slope-polystable with respect to the polarisation $H_{n-1}|_{X^{(n-2)}}$
 \item[(iii)] The double dual $F_i|_{X^{(n-2)}}^{\vee \vee}$ is isomorphic to  $E|_{X^{(n-2)}}$ for $i =1,2$.
 \item[(iv)] The sheaf $E|_{X^{(n-1)}}$ is polystable.
\end{enumerate}

It follows from the arguments on p.~223/224 of \cite{HL} or from a slight modification of the proof of Theorem~\ref{thm:semiampleness} that in our setup, spelled out under \emph{Simplifying Assumptions} above, there exist $r_0, r_1 \in \mathbb{N}^+$ such that for every $\nu \geq 1$ there is a linear ``lifting'' map 
\[H^0\bigl(Q_{X^{(n-1)}}, (\mathscr{L}_0^{(n-1)})^{\otimes \nu r_0}  \bigr)^{\SL(V^{(n-1)})} \to  H^0\bigl(S, \lambda_{\mathscr{F}^{(n-2)}} \bigl(u_1(c|_{X^{(2)}}) \bigr)^{\otimes \nu r_1} \bigr)^{\SL(V)},\]
which is induced by restricting the flat family $\mathscr{F}^{(n-2)}$ of sheaves on the surface $X^{(n-2)}$ further down to $S \times X^{(n-1)}$, and which we will call $l_{\mathscr{F}}^{X^{(n-1)}X^{(n-2)}}$.

As the $0$-cycles $C_{F_1}^{X^{(n-2)}}$ and $C_{F_2}^{X^{(n-2)}}$ on $X^{(n-2)}$ do not coincide, we deduce from items (ii)--(iii) above and from \cite[proof of Prop.~8.2.13 and discussion on the lower part of p.~228]{HL} 
that there exist finitely many curves $X^{(n-1)}_1, \dots, X^{(n-1)}_k \in U_2$, a positive natural number $\nu$, and 
sections $\sigma^{(n-1)}_i\in H^0(Q_{X^{(n-1)}_i}, (\mathscr{L}_0^{(n-1)})^{\nu r_1})^{\SL(V^{(n-1)})}$
such that the linear combination 
\[ \sigma^{(n-2)} \definiere  \sum_{i=1}^k  l_{\mathscr{F}}^{X^{(n-1)}_iX^{(n-2)}}(\sigma_i^{(n-1)}) \in H^0\bigl(S, \lambda_{\mathscr{F}^{(n-2)}} \bigl(u_1(c|_{X^{(2)}}) \bigr)^{\otimes \nu r_1} \bigr)^{\SL(V)} \]
separates $s_1$ and $s_2$; i.e., we have
\begin{equation}\label{eq:sigmaseparates}
 0=\sigma^{(n-2)}(s_1) \neq \sigma^{(n-2)}(s_2).
\end{equation}
In our setup, the prerequisites of Proposition~\ref{prop:determinantcomparison} are fulfilled, and hence, some positive tensor power of $\lambda_{\mathscr{F}^{(n-2)}} \bigl(u_1(c|_{X^{(2)}}) \bigr)$ is $\SL(V)$-equivariantly isomorphic to some (other) positive tensor power of $\mathscr{L}_{n-1}$. Via this $\SL(V)$-equivariant isomorphism, the section $\sigma^{(n-2)}$ induces an invariant section $\tau \in H^0\bigl(S, \mathscr{L}_{n-1}^{\otimes s} \bigr)^{\SL(V)}$ for some $s > 0$. Because of \eqref{eq:sigmaseparates}, we have $0 = \tau(s_1) \neq \tau(s_2)$, i.e., $\tau$ separates the two points $s_1$ and $s_2$. These are therefore mapped to different points by the morphism $\Phi\colon S \to M^{\mu ss}$, which implies that $F_1$ and $F_2$ give rise to different points in the moduli space $M^{\mu ss}$. This concludes the proof of Theorem~\ref{thm:sep}. 
\end{proof}

\begin{rem}
Note that in Case 2 of the previous proof the subcase when $\Supp\, C_{F_1}\neq \Supp\, C_{F_2}$ may be dealt with easily 
by using some complete intersection \emph{curve} meeting only one of the two supports. In the remaining subcase,
if $X^{(n-1)}$ meets $\Supp\, C_{F_1}= \Supp\, C_{F_2}$, then all lifted sections will vanish at $s_1$ as well as at $s_2$. As a consequence, the curves $X^{(n-1)}_1, \ldots, X^{(n-1)}_k \in U_2$ used in the proof will in general avoid $\Supp\, C_{F_1}= \Supp\, C_{F_2}$, and none of the individual sections $\sigma^{(n-1)}_i$ needs to vanish at $s_1$. This is why lifting from several curves $X^{(n-1)}_i$ is needed in order to achieve separation in this case.
\end{rem}

We now reduce the problem of identifying the equivalence relation realised by $M^{\mu ss}$ to a question concerning the connectedness of the fibres of the ``Quot to Chow'' morphism \cite{Fog69, Rydh}. 

\begin{Lem}\label{lem:non-separation}
Let $\mathscr{E}$ be a flat family of $(H_1,\ldots,H_{n-1})$-semistable sheaves on $X$ over a (connected) curve $S$, 
such that for all $s\in S$ one has $\mathscr{E}_s^{\sharp}\cong F$ and $C_{\mathscr{E}_s}=C$ 
for some fixed reflexive sheaf $F$ and some fixed cycle $C$ of codimension two on $X$. 
Then, some positive tensor power of the associated determinant line bundle $\mathscr{L}_{n-1}$ on $S$ is trivial.
\end{Lem}

\begin{proof}
The singular set $\Sing(\mathscr{E}_s)$ of each individual sheaf $\mathscr{E}_s$ consists of a codimension-two part, which is equal to $\Supp (C)$, and a part of higher codimension, which might depend on $s$. The union $B\definiere \bigcup_{s\in S}\Sing(\mathscr{E}_s)$ is thus a constructible set of codimension at least two in $X$.
We may therefore choose  a complete intersection curve $X^{(n-1)}$ that satisfies the conditions of our Proposition \ref{prop:determinantcomparison} and that avoids $B$. Moreover, we observe that $X^{(n-1)}$ may be chosen in such a way that in addition to the above $F|_{X^{(n-1)}}$ is a semistable locally free sheaf on $X^{(n-1)}$.  

With these choices, we have $\mathscr{E}_s|_{X^{(n-1)}}\cong F|_{X^{(n-1)}}$ for all $s \in S$, and 
the restricted family $\mathscr{E}|_{X^{(n-1)}\times S}$ of semistable sheaves on $X^{(n-1)}$ gives rise to a single point in the corresponding moduli space $M^{ss}_{X^{(n-1)}}\bigl(c|_{X^{(n-1)}},\det( F|_{X^{(n-1)}})\bigr)$ of semistable sheaves with the relevant Chern classes and with fixed determinant on the curve $X^{(n-1)}$. 
The ample line bundle $\overline{\mathscr{L}}_0$ on the moduli space $M^{ss}_{X^{(n-1)}}\bigl(c|_{X^{(n-1)}},\det( F|_{X^{(n-1)}})\bigr)$ described in \cite[Thm.~8.1.11]{HL} pulls back by the classifying morphism to a trivial line bundle on $S$, where it is linearly equivalent to the determinant line bundle $\lambda_{\mathscr{E}^{(n-1)}}({u}_0(c|_{X^{(n-1)}}))$ by \cite[Thm.~8.1.5(2)]{HL}. The assertion on $\mathscr{L}_{n-1}$ now follows from the isomorphism \eqref{eq:determinantcomparison} of Proposition \ref{prop:determinantcomparison}. 
\end{proof}

Lemma~\ref{lem:non-separation} implies the following non-separation criterion, which under an additional connectivity assumption gives a converse to Theorem~\ref{thm:sep}, and hence in this case gives a complete sheaf-theoretic description of the equivalence relation realised by $M^{\mu ss}$.
\begin{prop}[Non-separation]\label{prop:non-separation}
Let $F_1$, $F_2$ be two $(H_1,\ldots,H_{n-1})$-semi\-stable sheaves with the same Hilbert polynomial $P$ on $X$ such that
\begin{enumerate}
 \item[(i)] $F_1^{\sharp}\cong F_2^{\sharp}=:F$, and
 \item[(ii)] $C_{F_1}=C_{F_2}=:C$.
\end{enumerate}
Suppose in addition that the isomorphism classes of the sheaves $F/gr^{\mu}F_1$ and $F/gr^{\mu}F_2$ lie in the same connected component 
of the fibre over $C$ of the canonical morphism from the seminormalisation of the Quot scheme $\mathrm{Quot}(F, P_F-P)$ 
to the Chow variety $\mathrm{Chow}_{n-2}(X)$ of cycles of codimension two on $X$.
Then, 
$F_1$ and $F_2$ give rise to the same point in   
 $ M^{\mu ss}$. 
\end{prop}

\begin{rem}
The connectedness and even the irreducibility of the fibres of the $\mathrm{Quot}$-to-$\mathrm{Chow}$ morphism is known when $X$ is two-dimensional \cite{EllLehn}, but does not hold in general. In absence of such a connectedness result, 
our criterion only says that the number of points in $M^{\mu ss}$ which give the same tuple $(F^{\sharp},C_F)$ is finite 
and bounded by the number of connected components of the fibre over $C_F$ of the morphism from $\Quot(F^{\sharp}, P_{F^{\sharp}}-P_F)$ to $\Chow_{n-2}(X)$.
\end{rem}
It is an interesting problem to identify situations in which Proposition~\ref{prop:non-separation} hold unconditionally, which we will pursue in future work.

\subsection{$M^{\mu ss}$ as a compactification of the moduli space of $\mu$-stable reflexive sheaves}\label{subsect:compactificationofsimpl}
By work of Altman and Kleiman \cite[Thm.~7.4]{AltmanKleiman} (in the algebraic category) as well as by Kosarew and Okonek~\cite{KosarewOkonek} and Schumacher \cite{SchumacherKaehlerModuli} (in the analytic category), there exists a (possibly non-separated) coarse moduli space $M_{sim}$ for isomorphism classes of simple coherent sheaves on a fixed projective variety $X$. Since every $\mu$-stable sheaf is simple by \cite[Cor.~1.2.8]{HL}, it is a natural task to compare $M_{sim}$ with the newly constructed moduli space $M^{\mu ss}$. In fact, we will show that $M^{\mu ss}$ provides a natural compactification for the moduli space of $\mu$-stable reflexive sheaves (of fixed topological type) on $X$. 

Note that by \cite[Ch.~V, Thm.~2.8]{BSt_engl} and the characterization of reflexive sheaves given in \cite[Prop.~1.1.10(4)]{HL}, reflexivity is an open condition in the base space of a flat (proper) family. Let $M_{ \mathrm{refl}}^{\mu s}$ be the locally closed subscheme of $M_{sim}$ representing (isomorphism classes of) $\mu$-stable reflexive sheaves with class $c$ and determinant $\Lambda$. It follows for example from \cite[Prop.~6.6]{KosarewOkonek} and \cite[Cor.~7.12]{KobayashiVectorBundles} that $\big(M_{\mathrm{refl}}^{\mu s}\big)_{\mathrm{red}}$ is a separated quasi-projective variety. 

\begin{thm}[Compactifying the moduli space of stable reflexive sheaves]\label{thm:compactificationofsimple}
There exists a natural morphism 
$$\phi\colon \bigl(M_{ \mathrm{refl}}^{\mu s}\bigr)^{wn} \to M^{\mu ss} $$
that embeds $\bigl(M_{ \mathrm{refl}}^{\mu s}\bigr)^{wn}$ as a Zariski-open subset of $M^{\mu ss}$.
\end{thm}
\begin{proof}
The proof is divided into five steps.
\subsubsection*{Step 1: constructing the map $\phi$:}
As in Section~\ref{subsect:proof}, let $R^{\mu s}$ and $R^{\mu ss}$ denote the locally closed subschemes (of the Quot-scheme used to construct $M^{\mu ss}$) of all $\mu$-stable, or $\mu$-semistable quotients with the chosen invariants, respectively. Moreover, let $R^{\mu s}_{\mathrm{refl}}$ denote the subscheme of reflexive $\mu$-stable quotients. Furthermore, we let $S = (R^{\mu ss})^{wn}$ and $S^{\mu s}_{\mathrm{refl}} = (R^{\mu s}_{\mathrm{refl}})^{wn}$ be the respective weak normalisations, and we note that we have a natural $\SL(V)$-equivariant inclusion 
\begin{equation}\label{eq:inclusion}
  S^{\mu s}_{\mathrm{refl}} \hookrightarrow S.
\end{equation}
It follows from \cite[Lem.~4.3.2]{HL} that the centre of $\SL(V)$ acts trivially on all the spaces introduced above; i.e., the action of $\SL(V)$ factors over $\PGL(V)$, and moreover that the respective actions of $\PGL(V)$ on $R^{\mu s}$ and on $(R^{\mu s})^{wn}$ are set-theoretically free.

We will show that the $\PGL(V)$-action on $S^{\mu s}_{\mathrm{refl}}$ is proper. For this, it suffices to show that the action on $(S^{\mu s}_{\mathrm{refl}})^{an} =: \mathcal{S}$ is proper in the topological sense. As this action is set-theoretically free, it suffices to establish the following two properties: 
\begin{enumerate}
 \item[($\alpha$)] The quotient topology on $\mathcal{S}/\PGL(V)$ is Hausdorff, and 
 \item[($\beta$)] there exists \emph{local slices} through every point of $\mathcal{S}$; i.e., through every point $s \in \mathcal{S}$ there exists a locally closed analytic subset $s \in T \subset \mathcal{S}$ such that $\PGL(V)\acts T$ is open in $\mathcal{S}$ and such that the map $\PGL(V) \times T \to \PGL(V)\acts T \subset \mathcal{S}$ is biholomorphic.
\end{enumerate}
If we set $\mathcal{M}\definiere  \bigl((M_{ \mathrm{refl}}^{\mu s})^{wn}\bigr)^{an}$, then $\mathcal{M}$ is the (analytic) coarse moduli space for families of $\mu$-stable reflexive sheaves with the chosen invariants parametrised by weakly normal complex base spaces. Consequently, the restriction of the universal family from $R$ to $\mathcal{S}$ gives rise to a holomorphic classifying map $\pi\colon~\mathcal{S} \to \mathcal{M}$. Since isomorphism classes of sheaves parametrised by $\mathcal{S}$ are realised by the $\PGL(V)$-action, cf.~\cite[Sect.~4.3]{HL}, the map $\pi$ induces an injective continuous map $\mathcal{S}/\PGL(V) \to \mathcal{M}$. Since $\mathcal{M}$ is Hausdorff, $\mathcal{S}/\PGL(V)$ is likewise Hausdorff. This shows ($\alpha$).

Let now $s_0 \in \mathcal{S}$, and $\pi(s_0)$ the corresponding point in the moduli space $\mathcal{M}$. Then, we may find an open neighbourhood $W$ of $\pi(s_0)$ in $\mathcal{M}$ such that there exists a universal family $\mathscr{U}$ over $W \times X$, see~\cite[Thm.~6.4]{KosarewOkonek} or \cite[Thm.~7.4]{AltmanKleiman}. After shrinking $W$ if necessary, the family $\mathscr{U}$ induces a holomorphic section $\sigma\colon W \to \pi^{-1}(W) \subset \mathcal{S}$ of $\pi|_{\pi^{-1}(W)}$ through $s_0 \in \mathcal{S}$. Since every fibre of $\pi$ is a $\PGL(V)$-orbit, we conclude that $\PGL(V) \acts \sigma(W) = \pi^{-1}(W)$ is open in $\mathcal{S}$. Moreover, since $\pi|_{\sigma(W)}\colon \sigma(W) \to W$ is biholomorphic, hence bijective, and since $\mathcal{M}$ parametrises isomorphism classes of $\mu$-stable reflexive sheaves, for any $s \in \sigma(W)=:T$ we have $\PGL(V)\acts s \cap \sigma(W) = \{s\}$.
As a consequence, the natural map $\eta\colon \PGL(V) \times T \to \PGL(V)\acts T$ is holomorphic, open, and bijective. As $\PGL(V)\acts T \subset \mathcal{S}$ is weakly normal, $\eta$ is therefore biholomorphic; i.e., $T = \sigma(W)$ is a local holomorphic slice through $s_0$.

To sum up, we have established that $\PGL(V)$ acts properly on $S^{\mu s}_{\mathrm{refl}}$. As a consequence, the geometric quotient $S^{\mu s}_{\mathrm{refl}} / \PGL(V)$ exists in the category of algebraic spaces \cite{KollarQuotients}, and once the existence of this quotient is established, it is rather straightforward to see that in fact $(M^{\mu s}_{\mathrm{refl}})^{wn} \cong S^{\mu s}_{\mathrm{refl}} / \PGL(V)$. By abuse of notation, we will denote the corresponding quotient  map $S^{\mu s}_{\mathrm{refl}} \to (M^{\mu s}_{\mathrm{refl}})^{wn}$ by $\pi$.

Recalling the inclusion \eqref{eq:inclusion}, we may restrict the $\SL(V)$-invariant and hence $\PGL(V)$-invariant map $\Phi$ (cf.~Definition~\ref{defi:modulispace}) to $S^{\mu s}_{\mathrm{refl}}$. As $\pi$ is a categorical quotient, the resulting map descends to a regular morphism $\phi\colon (M_{ \mathrm{refl}}^{\mu s})^{wn} \to M^{\mu ss}$ completing the following diagram:
\begin{equation}\begin{gathered}
 \begin{xymatrix}
  { S^{\mu s}_{\mathrm{refl}}\ar@{^(->}[r] \ar[d]_{\pi} & S  \ar[d]^{\Phi} \\
 (M^{\mu s}_{\mathrm{refl}})^{wn} \ar[r]^>>>>>{\phi}&  M^{\mu ss}.
}
 \end{xymatrix}
 \end{gathered}
\end{equation}
This concludes the construction of the desired map from $(M^{\mu s}_{\mathrm{refl}})^{wn}$ to $M^{\mu ss}$.

\subsubsection*{Step 2: $\phi$ is injective:} This follows immediately from Lemma~\ref{lem:separation}.

\subsubsection*{Step 3: $\phi\big((M^{\mu s}_{\mathrm{refl}})^{wn}\big)$ is open:} The set $A: = S \setminus S^{\mu s}_{\mathrm{refl}}$ is an $\SL(V)$-invariant closed subvariety of $S$. It follows from Proposition~\ref{prop:completeimage} that its image $\Phi(A) \subset M^{\mu ss}$ is closed. Furthermore, as a consequence of Lemma~\ref{lem:separation} we deduce that $A$ is $\Phi$-saturated, i.e., $\Phi^{-1}(\Phi (A)) = A$. Consequently, the set $$U \definiere  \phi\big((M_{\mathrm{refl}}^{\mu s})^{wn}\big) = \Phi\big(S^{\mu s}_{\mathrm{refl}}\big) = M^{\mu ss} \setminus \Phi(A)$$ is open, as claimed. 

\subsubsection*{Step 4: $\phi$ is open as a map onto its image $U$:} Since we have already seen that $\phi\colon (M^{\mu s}_{\mathrm{refl}})^{wn} \to U$ is bijective, it suffices to show that $\phi\colon (M^{\mu s}_{\mathrm{refl}})^{wn} \to U$ is closed. Let $\hat Z$ be any closed subvariety of $(M^{\mu s}_{\mathrm{refl}})^{wn}$. Then, let $Z \subset S^{\mu s}_{\mathrm{refl}}$ be its preimage under $\pi$, and $\overline{Z}$ the closure of $Z$ in $S$, which is automatically $\SL(V)$-invariant. As a consequence of Proposition~\ref{prop:completeimage}, the image $\Phi(\overline{Z})$ is closed in $M^{\mu ss}$. Hence, $\phi(\hat Z) = \Phi(\overline{Z}) \cap U$
is closed in $U$, as claimed. 

\subsubsection*{Step 5: conclusion of proof:} Summarising the previous steps, we know that $\phi\colon (M^{\mu s}_{\mathrm{refl}})^{wn} \to U$ is a bijective open morphism. Hence, its (set-theoretical) inverse $\phi^{-1}$ is continuous. Since $U \subset M^{\mu s}$ is weakly normal by construction, it follows that $\phi^{-1}$ is regular, and hence that $\phi$ is an isomorphism, as claimed. 
\end{proof}

\begin{rem}\label{rem:answeringTelemansQuestion}
 Theorem~\ref{thm:compactificationofsimple} above together with Proposition~\ref{prop:a_cic_in_every_chamber} below solves an old problem (raised for example independently by Tyurin and Teleman \cite[Sect.~3.2, Conj.~1]{Tel08}) of exhibiting a sheaf-theoretically and geometrically meaningful compactification of the gauge-theoretic moduli space of vector bundles that are slope-stable with respect to a chosen K\"ahler class $[\omega]$ on a given K\"ahler manifold, in the particular case of projective manifolds and classes $[\omega] \in \mathrm{Amp}(X)_{\mathbb{R}}$.
\end{rem}

\subsection{Comparing $M^{\mu ss}$ with the Gieseker-Maruyama moduli space}\label{slopevsGM}
Up to this point, we have considered multipolarisations $(H_1, \ldots, H_{n-1})$ made up of possibly different ample line bundles $H_i$. In this section, we investigate the special case where \begin{equation}\label{eq:alleHsgleich}
 H_1 = H_2 = \ldots = H_{n-1} =: H.
\end{equation} In this setup, a compactification of the space of $H$-slope-stable vector bundles has long been known to exist, the so-called \emph{Gieseker-Maruyama moduli space}, see \cite[Chap.~4]{HL} and the references given there. The following result compares this Gieseker-Maruyama moduli space with our newly constructed moduli space $M^{\mu ss}$. 
\begin{prop}[Comparing $M^{Gss}$ and $M^{\mu ss}$]\label{prop:Giesekerslopecomparison}
 Let $X$ be a projective manifold of dimension $n$ with ample line bundle $H$, let $M^{\mu ss} = M^{\mu ss}(c, \Lambda)$ be the moduli space of $\mu_H$-semistable sheaves with Chern classes given by $c \in K(X)_{num}$ and fixed determinant line bundle $\Lambda$, and let $M^{Gss} = M^{Gss}(c, \Lambda)$ be the Gieseker-Maruyama moduli space for sheaves with the same invariants. Then, the following holds.
\begin{enumerate}
 \item There exist a line bundle $\lambda_{M^{Gss}}\bigl(u(c)\bigr)$ on $M^{Gss}$, unique up to isomorphism, such that for every flat family $\mathscr{E}$ of Gieseker-semistable sheaves parametrised by a scheme $Z$, with associated classifying morphism $\Psi_\mathscr{E}\colon Z \to M^{Gss}$, we have
\[\Psi_\mathscr{E}^*\bigl(\lambda_{M^{Gss}}\bigl(u(c)\bigr)\bigr) \cong \lambda_\mathscr{E} \bigl(u_{n-1}(c)\bigr). \] 
In the presence of group actions on $Z$ and $\mathscr{E}$, the previous isomorphism is an isomorphism of linearised line bundles.

 We will denote the pullback of the line bundle $\lambda_{M^{Gss}}\bigl(u(c)\bigr)$ to the weak normalisation $(M^{Gss})^{wn}$ by $\overline{\mathscr{L}}_{n-1}$.
 \item The line bundle $\overline{\mathscr{L}}_{n-1}$ is semiample. 
We set \[M^{Gss}_{n-1} \definiere  \mathrm{Proj}\Big(\bigoplus_{k\geq 0} H^0\bigl((M^{Gss})^{wn}, \overline{\mathscr{L}}_{n-1}^{\otimes k} \bigr)\Big),\] and we denote the natural morphism from $(M^{Gss})^{wn}$ to $M^{Gss}_{n-1}$ by $\Phi^{Gss}_{n-1}$. 
 \item If $N \in \mathbb{N}^+$ is as in the Main Theorem, and $S$ is as defined in \eqref{eq:Sdefined} above, the restriction of the universal family $\mathscr{F}$ from $S$ to to the open subset $S^{Gss}\definiere  \{s \in S \mid \mathscr{F}_s \text{ is Gieseker-semistable}\}$ induces a morphism $\overline \Phi\colon(M^{Gss})^{wn} \to M^{\mu ss}$ such that 
\begin{equation}\label{eq:barPhipullback}
 \overline{\Phi}^*(\mathscr{O}_{M^{\mu ss}}(1)) \cong \overline{\mathscr{L}}_{n-1}^{\otimes N}.
\end{equation}
 The Stein factorisation of $\overline{\Phi}$ is given by the following commutative diagram:
 \begin{equation}
  \begin{gathered}
   \begin{xymatrix}{
    (M^{Gss})^{wn} \ar[rd]^{\overline{\Phi}} \ar[d]_{\Phi^{Gss}_{n-1}} &  \\
    M^{Gss}_{n-1} \ar[r]^>>>>>>{\eta} & M^{\mu ss}.
}
   \end{xymatrix}
  \end{gathered}
 \end{equation}
\end{enumerate}
\end{prop}
\begin{proof} If we set $h = [\mathscr{O}_H] \in K(X)$, owing to \eqref{eq:alleHsgleich} we have 
\[u_{n-1}(c) = -r h^{n-1} + \chi(c\cdot h^{n-1})[\mathscr{O}_x] \in K(X).\]
Hence, item (1) is a direct consequence of \cite[Lem.~3.1 and Thm.~2.5(2)]{LePotierDonaldsonUhlenbeck} or \cite[Thm.~8.1.5]{HL}.

Next, we prove the semiampleness claim of item (2). For this, we will use the setup and notation of Section~\ref{subsect:proof}. It follows from the construction of the Gieseker-Maruyama moduli space, as carried out for example in \cite[Sect.~5]{HL}, that, possibly after increasing the multiple of the ``twisting'' line bundle $\mathscr{O}_X(1)$, the following lemma holds. 
\begin{lemma}[Realising the Giesker-Maruyama moduli space as good quotient]
 The good quotient $\Pi\colon R^{Gss} \to R^{Gss}\hq \SL(V)$ of the $\SL(V)$-invariant open subscheme $R^{Gss}$ that parametrises Gieseker-semistable sheaves in $R^{\mu ss}$ exists. Moreover, the quotient $R^{Gss}\hq \SL(V)$ is isomorphic to the Gie\-se\-ker-Maruyama moduli space $M^{Gss}$. 
\end{lemma}
 As the weak normalisation of $R^{Gss}$ is isomorphic to $S^{Gss}= \{s \in S \mid \mathscr{F}_s \text{ is Gieseker-semistable}\}$, exploiting the fundamental property of the weak normalisation, see Proposition~\ref{propnot:weaknormalisation}, for the composition $S^{Gss} \to R^{Gss} \overset{\Pi}\longrightarrow M^{Gss}$ we obtain a natural $\SL(V)$-invariant morphism $\pi\colon S^{Gss} \to (M^{Gss})^{wn}$, which fits into the following commutative diagram:
\[\begin{xymatrix}
   {
S^{Gss} \ar[r]\ar[d]_{\pi} & R^{Gss} \ar[d]^{\Pi} \\
(M^{Gss})^{wn} \ar[r] & M^{Gss}.
}
  \end{xymatrix}
\]
Here, the horizontal maps are the respective weak normalisation morphisms. Since $\Pi$ is a good quotient, the map $\pi$ is a good quotient for the action of $\SL(V)$ on $S^{Gss}$. Moreover, it is the classifying morphism for the family of Gieseker-semistable sheaves $\mathscr{F}^{Gss} = \mathscr{F}|_{S^{Gss} \times X}$. Consequently, item (1) yields a $\SL(V)$-equivariant isomorphism 
\begin{equation}\label{eq:Giesekerpullbackiso}
 \mathscr{L}_{n-1}|_{S^{Gss}} \cong \pi^*\bigl(\overline{\mathscr{L}}_{n-1}\bigr).
\end{equation}
It follows from Theorem~\ref{thm:semiampleness} that there exists a natural number $n_0 \in \mathbb{N}^+$ such that $\mathscr{L}_{n-1}^{\otimes k_0}$ is generated by $\SL(V)$-invariant sections over $S$. Hence, in particular, $\mathscr{L}_{n-1}^{\otimes n_0}|_{S^{Gss}}$ is generated by $\SL(V)$-invariant sections over $S^{Gss}$. As $\pi$ is a good quotient, the equivariant isomorphism \eqref{eq:Giesekerpullbackiso} together with the equivariant projection formula implies that $\pi_* (\mathscr{L}_{n-1}^{\otimes n_0}|_{S^{Gss}})^{\SL(V)} \cong \overline{\mathscr{L}}_{n-1}^{\otimes n_0}$, and hence that $H^0\bigl(S^{Gss}, \mathscr{L}_{n-1}^{\otimes n_0}|_{S^{Gss}} \bigr)^{\SL(V)} \cong H^0\bigl(M^{Gss}, \overline{\mathscr{L}}_{n-1}^{\otimes n_0} \bigr)$.
Semiampleness of $\overline{\mathscr{L}}_{n-1}$ follows.

Finally, we prove the claims made in item (3). By the universal property of $M^{\mu ss}$ the $\SL(V)$-equivariant family $\mathscr{F}^{{Gss}}$ induces an $\SL(V)$-invariant morphism $S^{Gss} \to M^{\mu ss}$, which coincides with $\Phi|_{S^{Gss}}$. Since $(M^{Gss})^{wn} = S^{Gss}\hq \SL(V)$, this morphism admits a factorisation as 
\begin{equation}\label{eq:defofcircPhi}
 \Phi|_{Gss} = \overline{\Phi} \circ \pi,
\end{equation}
 where  $\overline{\Phi}\colon M^{Gss} \to M^{\mu ss}$ is the desired morphism. It follows from the universal properties of the triple $(M^{\mu ss}, \mathscr{O}_{M^{\mu ss}}(1), N)$ and from the equivariant isomorphism \eqref{eq:Giesekerpullbackiso} that there are isomorphisms $\Phi|_{S^{Gss}}^*\bigl(\mathscr{O}_{M^{\mu ss}}(1)\bigr) \cong \mathscr{L}_{n-1}^{\otimes N}|_{S^{Gss}} \cong \pi^*(\overline{\mathscr{L}}_{n-1}^{\otimes N})$ that are compatible with the respective $\SL(V)$-linearisations. As $\pi_*(\mathscr{O}_{S^{Gss}})^{\SL(V)} = \mathscr{O}_{(M^{Gss})^{wn}}$, equation \eqref{eq:defofcircPhi} and the equivariant projection formula consequently imply $\overline{\Phi}^*(\mathscr{O}_{M^{\mu ss}}(1)) \cong \overline{\mathscr{L}}_{n-1}^{\otimes N}$,
as claimed. 

For the discussion of the remainig claims made in item (3), let
\begin{equation}\label{eq:Steinfactorisation}
 M^{Gss}  \overset{\widehat{\Phi}}{\longrightarrow} \widehat{M}^{\mu ss}{\longrightarrow} M^{\mu ss}
\end{equation}
be the Stein factorisation of $\overline{\Phi}$. We have to show the existence of an isomorphism $f\colon M^{Gss}_{n-1} \to \widehat{M}^{\mu ss}$ that completes diagram \eqref{eq:Steinfactorisation} as follows
\begin{equation}\label{eq:completedSteinfactorisation}
\begin{gathered}
\begin{xymatrix}{
 &  M^{Gss} \ar[ld]_{\Phi^{Gss}_{n-1}} \ar[d]_{\widehat{\Phi}} \ar[rd]^{\overline{\Phi}} &    \\
M^{Gss}_{n-1} \ar[r]^>>>>>{f}_>>>>>{\cong} &  \widehat{M}^{\mu ss} \ar[r]& M^{\mu ss}.
}
\end{xymatrix}
\end{gathered}
\end{equation}
 We note that the map $\widehat \Phi$ is given by a subalgebra of $R(M^{Gss}, \overline{\mathscr{L}}_{n-1})$. Hence, the universal properties of the $\mathrm{Proj}$-construction yield a morphism $f\colon M^{Gss}_{n-1} \to \widehat{M}^{Gss}$ that makes diagram \eqref{eq:completedSteinfactorisation} commutative.

It remains to construct an inverse morphism $g\colon \widehat{M}^{Gss} \to M^{Gss}_{n-1}$. To achieve this task, by the universal property of the Stein factorisation, see \cite[Chap.~10, \S 6]{CAS}, it suffices to show that $\Phi^{Gss}_{n-1}$ is constant on the connected fibre components of $\overline{\Phi}$. So, let $C$ be a connected projective curve $C \subset M^{Gss}$ lying in a fibre of $\overline{\Phi}$. As a consequence of \eqref{eq:barPhipullback} the restriction $\overline{\mathscr{L}}_{n-1}^{\otimes N}|_C$ is trivial. Hence, it follows from the functorial properties of the Proj-construction that $C$ is contracted to a point by $\Phi^{Gss}_{n-1}$.
\end{proof}  
\begin{rem}
The map $\overline \Phi$ is birational when restricted to the respective closures of the weak normalisation $(M^{\mu s}_{\mathrm{refl}})^{wn}$ of the moduli space of $\mu$-stable reflexive sheaves, which embeds into both $(M^{Gss})^{wn}$ and $M^{\mu ss}$, cf.~Section~\ref{subsect:compactificationofsimpl}.
\end{rem}

\begin{rem} Note that in the surface case $n=2$ the starting point of Le Potier \cite{LePotierDonaldsonUhlenbeck} and Li \cite{JunLiDonaldsonUhlenbeck} is to study the line bundle $\lambda_{M^{Gss}}\bigl(u_1(c)\bigr)$ on $M^{Gss}$. Both authors show that this bundle is semiample. Le Potier \cite[Sect.~4]{LePotierDonaldsonUhlenbeck} then studies the maps given by complete linear systems of sections in high powers of $\overline{\mathscr{L}}_{n-1}$, whereas Li \cite[Sect.~3]{JunLiDonaldsonUhlenbeck} focusses on linear systems of sections that are lifted from curves. Later, Huybrechts and Lehn \cite[Chap.~5]{HL} introduce an approach that does not restrict to families of Gieseker-semistable sheaves, but more generally considers families of slope-semistable sheaves. It follows from the description of the resulting moduli spaces that the slightly different approaches of Le Potier, Li, and Huybrechts-Lehn induce the same equivalence relation on $M^{Gss}$, and are hence equivalent.
\end{rem}

\section{Wall-crossing problems} \label{sect:wall-crossing}
In the present section, we investigate wall-crossing questions for moduli spaces of sheaves on higher-dimensional base manifolds. The results obtained here, especially Theorem~\ref{thm:chambers} and Proposition~\ref{prop:a_cic_in_every_chamber}, are one of the main motivations for constructing the moduli space of $(H_1, \ldots, H_{n-1})$-semistable sheaves, as carried out in the previous sections of this paper.

\subsection{Motivation -- The work of Qin and Schmitt}\label{subsect:motivation}As sketched in the Introduction, there is a well-developed theory for wall-crossing phenomena of moduli spaces of sheaves on \emph{surfaces}. Investigating these phenomena for moduli spaces of Gieseker-semistable sheaves over \emph{higher-dimensional base manifolds}, Qin adapts his notion of ``wall'' from the two-dimensional to the higher-dimensional case. But in contrast to the surface case, he immediately finds examples of varieties (with Picard number three) where these walls are not locally finite inside the ample cone, see \cite[Ex.~I.2.3]{Qin93}. 

In order to avoid these pathologies, Schmitt \cite{Sch00} restricts his attention to segments inside the ample cone that connect integral ample classes. Provided that wall-crossing occurs on a rational wall, he is able to derive results that are similar in spirit those to those obtained for two-dimensional base manifolds by Matsuki and Wentworth \cite{MatsukiWentworth}. However, he also gives examples of three\-folds where this condition is not satisfied. More precisely, he exhibits threefolds $X$ with Picard number equal to two carrying a rank two vector bundles $E$ that are $\mu$-stable with respect to some integral ample divisor $H_0$ and unstable with respect to some other integral ample divisor $H_1$ such that the class $H_{\lambda}\definiere 
 (1-\lambda)H_0+\lambda H_1$ for which $E$ becomes strictly semistable is irrational, see \cite[Ex.~1.1.5]{Sch00}. This irrationality can be traced back to the fact that $\lambda\in \R$ is obtained as the solution of a quadratic equation given by a condition of the form $H^2_{\lambda}D=0$ for a suitable rational divisor $D$, cf.~Section~\ref{sect:Qin_and_Schmitt}.

In the subsequent sections, we will solve these problems based on the philosophy that the natural ``polarisations'' to consider when defining slope-semi\-sta\-bi\-lity on higher-dimensional base manifolds are not ample divisors but rather movable curves.

Recalling some notions already introduced in Section~\ref{subsect:semistability}, given an $n$-di\-men\-sional smooth projective variety $X$, let $N_1=N_1(X)_{\R}$ be the space of $1$-cycles with real coefficients modulo numerical equivalence, and let $N^1=N^1(X)_{\R}$ be the dual space of divisor classes, which contains the open cone $\mathrm{Amp}(X)$ of real ample divisor classes. Consider the associated subset $P(X)$ of $N_1$ consisting of $(n-1)$-st powers of real ample classes in $N^1$, which is contained inside the cone spanned by classes of movable curves.

We prove that $P(X)$ is open in $N_1$, and that the natural map $\mathrm{Amp}(X) \to P(X)$ (taking $(n-1)$-st powers) is an isomorphism, see Proposition~\ref{P}. More\-over, we show in Theorem~\ref{thm:chambers} that $P(X)$ supports a locally finite chamber structure given by linear rational walls such that the notion of slope-semistability is constant within each chamber. Furthermore, every chamber (even if it is not open) contains products  $H_1H_2...H_{n-1}$ of integral ample divisor classes, see 
Proposition~\ref{prop:a_cic_in_every_chamber}.

\subsection{Boundedness} \label{subsect:boundedness}
The current section is devoted to the proof of some preparatory boundedness statements. Since there is no added complication, in these preparatory statements we consider slope-semistability with respect to arbitrary K\"ahler classes. We say that a torsion-free sheaf is (semi)stable with respect to some K\"ahler class $\omega$ if it is (semi)stable with respect to $\omega^{n-1}$, cf.~Section~\ref{subsect:semistability}.

We start by adapting a preliminary boundedness result from \cite{Tel08} to our situation.

\begin{Lem}\label{compacity}
Let $X$ be a compact complex manifold of dimension $n$ endowed with a K\"ahler class $\phi$, and let $E$ be a torsion-free
sheaf on $X$. Then, for every $d\in \R$ the set $B\!N(E)_{\ge d}\definiere \{ L\in \mathrm{Pic}(X) \ | \ \mathrm{Hom}(L,E)\neq 0, \  \deg_{\phi}L\ge d\}$ is compact in $\mathrm{Pic} (X)$. 
\end{Lem}
\begin{proof}
By the flattening result of Raynaud  \cite{Ray72}  and Hironaka's theorem on elimination of points of indeterminacy \cite{Hir64}, there exists a composition of blow-ups with smooth centres $f\colon X'\to X$ such that $E'\definiere f^*E/\mathscr{T}\negthinspace or(f^*E)$ is locally free. Denote by $E_1$,..., $E_k$ the irreducible components of the exceptional divisor of $f$ in $X'$.
Then, $\phi'\definiere f^*\phi-a_1 [E_1]-a_2 [E_2]-...-a_k[E_k] \in N^1(X')_\mathbb{R}$ is a K\"ahler class on $X'$ for suitably chosen (small) positive real numbers $a_i$, and we shall compute degrees on $X'$ with respect to this class. 
Now, if $L$ belongs to $B\!N(E)_{\ge d}$ then $f^*L\in B\!N(E')_{\ge d}$.

To rephrase, we have just shown that the image of $B\!N(E)_{\ge d}$ under the injective holomorphic map $f^*\colon \mathrm{Pic}(X) \to \mathrm{Pic}(X')$ is contained in $B\!N(E')_{\ge d}$. This latter set is compact by \cite[Prop.~2.5]{Tel08}, and hence bounded. Consequently $B\!N(E)_{\ge d}$ is likewise bounded. Additionally, using Grauert's Semicontinuity Theorem one sees that $B\!N(E)_{\ge d}$ is closed in $\mathrm{Pic}(X)$. As $X$ is K\"ahler, it follows that $B\!N(E)_{\ge d}$ is compact, as claimed.\end{proof}

\begin{Lem}\label{properlyss}
Let $X$ be a compact complex manifold of dimension $n$ and let $\phi_0, \phi_1 \in H^{1,1}(X, \mathbb{R})$ be two K\"ahler classes on $X$. For every $\tau\in [0,1]$ we set 
$$\phi_{\tau}\definiere  (1-\tau)\phi_0+\tau \phi_1 \in H^{1,1}(X, \mathbb{R}).$$ Suppose that the torsion-free sheaf $E$ on $X$ is semistable with respect to $\phi_1$ and unstable  with respect to $\phi_0$, and let
$$t\definiere \inf\{\tau >0 \mid E {\rm  \ is \ semistable \ with \ respect \ to} \ \phi_{\tau}\}.$$ Then, $E$ is properly semistable with respect to $\phi_t$.
\end{Lem}
\begin{proof}
The continuity of $\tau\mapsto \deg_{\phi_\tau}(E)$ (\cite[Cor.~2.4]{Miyaoka}) implies 
 that $E$ is $\phi_t$-semistable. 
We shall  show that it is not $\phi_t$-stable. By the definition of $t$ there exists an integer $k$ with $0<k<\mathrm{rk}\, E$ and an increasing sequence $(\tau_n)_{n\in \N}$ of real numbers converging to $t$ together with a corresponding sequence $(F_n)$ of $\phi_{\tau_n}$-destabilising rank $k$ subsheaves of $E$. Then, for each $n \in \mathbb{N}$, the reflexive exterior power $(\bigwedge^k F_n)^{\vee \vee}$ is an invertible subsheaf of  $(\bigwedge^k E)^{\vee \vee}$ that is destabilising with respect to $\phi_{\tau_n}$. 
By Lemma \ref{compacity}, applied to $\phi=\phi_t$ and $d= \mu_{\phi_t}(\bigwedge^k F_n)^{\vee \vee}-1$, and by our assumption on $t$ we obtain an invertible subsheaf $L$ of $(\bigwedge^k E)^{\vee \vee}$ whose slope with respect to $\phi_t$ equals the slope of $(\bigwedge^k E)^{\vee \vee}$ by continuity of the degree. Without loss of generality, we may assume that $L$ is a saturated subsheaf of $(\bigwedge^k E)^{\vee \vee}$.

We will show that $L=(\bigwedge^k F)^{\vee \vee}$ for some subsheaf $F$ of $E$, which will have the same slope with respect to $\phi_t$ as $E$, which therefore is properly $\phi_t$-semistable. 
For this we follow some ideas contained in \cite[Sect.~2.2]{Tel08}. Without loss of generality, we may assume that $E$ is reflexive, and hence that there exists a Zariski-open subset $U$ of $X$ with $\mathrm{codim}_X(X\setminus U) \geq 2$ such that $E|_U$ is locally free. Moreover, since $L$ is saturated in $(\bigwedge^k E)^{\vee \vee}$, reflexivity of $E$ implies there exists an open subset $U' \subset U$ with $\mathrm{codim}_X(X\setminus U') \geq 2$ such that $L|_{U'}$ is a line subbundle of $(\bigwedge^k E)^{\vee \vee}|_{U'}$. Let $C_k(E|_{U'})\subset \bigwedge^k E_{U'}$ be the cone subbundle over the relative Grassmannian $G_s(E|_{U'})$. Then, over $U'$ the line bundle $L$ is contained in $C_k(E)$ and thus gives rise to a subbundle $F'$ of $E|_{U'}$ via projection to the relative Grassmannian, cf.~\cite[first paragr.~of Sect.~2.2]{Tel08}. If $F$ is the unique extension of $F'$ as a reflexive subsheaf of $E$, then we have $L = \bigwedge^k(F)^{\vee \vee}$, as desired. Moreover, by continuity we have $\mu_{\phi_t}(F)\ge \mu_{\phi_t}(E)$, which was to be shown.
\end{proof}

The next proposition proves boundedness of the set of torsion-free sheaves that are slope-semistable with respect to K\"ahler classes. Although we do not need the result in this generality, the techniques involved in the proof will be needed in the special case of polarisations from $P(X)$. We note that Proposition~\ref{prop:slopeboundedness} only applies to multipolarisations, and hence does not cover the case of polarisations from $P(X)$ needed here.
\begin{Prop}[Boundedness]\label{boundedness}
Let $X$ be a projective manifold of dimension $n$, and let $K$ a compact subset of the K\"ahler cone $\mathcal{K}(X) \subset H^{1,1}(X, \mathbb{R})$ of $X$. Fix a natural number $r>0$ and classes $c_i\in H^{2i}(X,\R)$. Then, the family
of rank $r$ torsion-free sheaves $E$ with $c_i(E)=c_i$ that are semistable with respect to some polarisation contained in $K$ is bounded.
\end{Prop}
\begin{proof} 
Let $\phi_1$ be some element of $K$. Choose an ample class $\phi_0$ in $H^{1,1}(X, \mathbb{R})$. For $\tau\in [0,1]$, set $\phi_{\tau}\definiere (1-\tau)\phi_0+\tau\phi_1$, and denote by 
$\mu_{\tau} $ the slope with respect to $\phi_{\tau}$. We shall prove boundedness by applying \cite[Thm.~3.3.7]{HL} with respect to the ample polarisation $\phi_0$. In order to establish the assertion of Proposition~\ref{boundedness}, it thus suffices to establish the following. 

\begin{Lem}\label{lem:boundedslope} In the setup of Proposition~\ref{boundedness}, if $E_{\mathrm{max}}$ is the maximally $\phi_0$-de\-stabi\-lising subsheaf of a $\phi_1$-semistable torsion-free sheaf
$E$ of rank $r$, then $\mu_0(E_{\mathrm{max}})$ is bounded by a constant depending only on $c_1(E)$, $c_2(E)$, $\phi_0$, and $K$.
\end{Lem}

\begin{proof}
The idea of the proof is to produce a filtration of $E$ such that the associated graduation has $\phi_0$-semistable terms whose slope with respect to $\phi_0$ is bounded by a constant depending only on $c_1(E)$, $c_2(E)$, $\phi_0$, and $K$. This constant will then bound $\mu_0(E_{\mathrm{max}})$ as well.
 
To implement this idea, let  $E_0=E$ be $\phi_1$-semistable and set $$t_1\definiere \inf\{\tau >0 \ | \ E_0 {\rm  \ is \ semistable \ with \ respect \ to} \ \phi_{\tau}\}.$$ If $t_1=0$, then $E$ is $\phi_0$-semistable, 
and our claim is verified. In the following argumentation we will therefore assume that $t_1 > 0$. Under this assumption, we know from Lemma~\ref{properlyss} that $E_0$ can be written as an extension
\begin{equation}\label{eq:ses}
 0 \to E_1 \to E_0\to E_2 \to 0
\end{equation}
with a torsion-free subsheaf $E_1$ and torsion-free quotient
$E_2$, such that $E_1$ and $E_2$ are both $\phi_{t_1}$-semistable with slopes 
\begin{equation}\label{eq:slopeequality}
 \mu_{t_1}(E_1)=\mu_{t_1}(E_2)=\mu_{t_1}(E_0).
\end{equation}
We define 
\begin{align}\label{eq:a1def}
 a_1 &\definiere \frac{\rk(E_2)c_1(E_1)-\rk (E_1)c_1(E_2)}{\rk (E_0)},
\end{align}
and we note that \eqref{eq:ses} taken together with the additivity of Chern classes in short exact sequences leads to
\begin{equation}\label{eq:newa1expression}
 a_1 = \frac{\rk(E_2)c_1(E_0)}{\rk (E_0)}-c_1(E_2) =c_1(E_1)
-\frac{\rk(E_1)c_1(E_0)}{\rk (E_0)}.
\end{equation}
We first establish a lower bound for the intersection of $- a_1^2$ with $\phi_{t_1}^{n-2}$. As a consequence of \eqref{eq:slopeequality} and \eqref{eq:newa1expression}  we obtain
\begin{equation}\label{eq:aisinkernel}
 a_1\phi_{t_1}^{n-1}=0,
\end{equation}
 i.e., $a_1$ is $\phi_{t_1}$-primitive. The Hodge Index Theorem 
 implies that $a \mapsto -a ^2\phi_{t_1}^{n-2}$ defines the square of a norm on $\mathrm{Ker}(\phi_{t_1}^{n-1})\subset NS(X)_{\mathbb{R}}$. In particular, from \eqref{eq:aisinkernel} we infer that 
\begin{equation}\label{eq:lowerbound}
 0 \leq-a_1 ^2\phi_{t_1}^{n-2}.
\end{equation}
Moreover, equality in \eqref{eq:lowerbound} is achieved if and only if $a_1=0$.

With these preparations in place, we will show that $-a_1^2\phi_{t_1}^{n-2}$ is bounded from above by some constant that only depends on $c_1(E)$, $c_2(E)$, $\phi_1$ and $K$.  Once this bound is established, we will conclude that $a_1$ is contained in a finite set that depends only on $c_1(E)$, $c_2(E)$, $\phi_0$, and $K$. The equality \eqref{eq:newa1expression} will then give the desired bound on $c_1(E_1)$, and consequently also on $c_1(E_2)$ (note that $\phi_{t_1}$ belongs to the convex hull of $K$ and $\phi_0$). 

Before we proceed, recall the definition of the discriminant of a torsion-free sheaf $F$, cf.~\cite[Sect.~3.4]{HL}: 
\begin{equation} \label{eq:discrdef}
 \Delta (F)=\frac{1}{\rk (F)}\bigl(c_2(F)-\frac{\rk(F)-1}{2\rk (F)}c_1^2(F)\bigr).
\end{equation}

A short computation using \eqref{eq:discrdef} and \eqref{eq:newa1expression} shows that we can express the discriminant of $E_0$ in terms of $a_1$ and in terms of the discriminants of $E_1$ and $E_2$, as follows:
\begin{equation}\label{eq:alternativediscriminant}
 \Delta (E_0)=-\frac{1}{2\rk( E_1) \rk(E_2)}\,a_1^2+\frac{\rk (E_1)}{\rk (E_0)}\,\Delta(E_1) + \frac{\rk (E_2)}{\rk (E_0)}\,\Delta (E_2).
\end{equation}
Since both $E_1$ and $E_2$ are $\phi_{t_1}$-semistable, the Bogomolov inequality (see \cite[Cor.~3]{BandoSiu} for the case of polystable reflexive sheaves and \cite[Lem.~2.1]{BiswasMcKay} for the general case) holds for both sheaves; i.e., we have
\begin{equation}\label{eq:Bogomolov}
 \Delta(E_i)\phi_{t_1}^{n-2}\ge 0 \quad \quad \text{for }i=1,2.
\end{equation} 
Combining the lower bound \eqref{eq:lowerbound} with the expression \eqref{eq:alternativediscriminant} and the Bogomolov inequalities \eqref{eq:Bogomolov} we infer that $$0 \leq -\frac{1}{2\,\rk( E_1)\, \rk(E_2)}\,a_1^2\phi_{t_1}^{n-2}\le \Delta (E_0)\phi_{t_1}^{n-2},$$
which establishes the desired bound for $a_1$, since $\phi_{t_1}$ lies in the compact set $\{(1-\tau) \phi_0 + \tau \phi_1 \mid \tau \in [0,1], \phi_1 \in K \}$.

We now iterate this argument. For this, we set $$t_2\definiere \inf\{\tau >0 \ | \ E_1, \ E_2  {\rm  \ are \ semistable \ with \ respect \ to} \ \phi_{\tau}\}.$$
 If $t_2=0$ we are done as before, for $0\subset E_1\subset E$ is the desired filtration.
 
When  $t_2 \neq 0$ 
 one of  $E_1$, $E_2$ will be properly $ \phi_{t_2}$-semistable and the other $ \phi_{t_2}$-semistable.
 For simplicity of notation
  suppose that $E_2$ is properly $ \phi_{t_2}$-semistable and denote by $E_3$ a subsheaf of $E_2$ with
 torsion-free quotient $E_4$ such that 
 $E_3$ and $E_4$ are $\phi_{t_2}$-semistable and $\mu_{t_2}(E_3)=\mu_{t_2}(E_4)=\mu_{t_2}(E_2)$.

 We will use the following shorthand notation: $r_i  \definiere \rk (E_i)$, $\frak{c}_i \definiere c_1(E_i)$, and $\Delta_i \definiere \Delta(E_i)$. In analogy with the definition of $a_1$, cf.~\eqref{eq:a1def}, we set
\begin{align*}
 a_2 \definiere \frac{r_4\frak{c}_3 - r_3\frak{c}_4}{r_2} = \frac{r_4\frak{c}_2}{r_2}-\frak{c}_4 = \frak{c}_3 - \frac{r_3\frak{c}_2}{r_2}
\end{align*}

 As in the first step we see that $a_2$ is $\phi_{t_2}$-primitive. Furthermore, comparing discriminants we arrive at
\begin{align*}
 \Delta_0 + \frac{1}{2r_1r_2}\,a_1^2 &= \frac{r_1}{r_0}\Delta_1 + \frac{r_2}{r_0}\Delta_2 =
  \frac{r_1}{r_0}\Delta_1 + \frac{r_2}{r_0}\cdot \bigg(- \frac{1}{2r_3r_4}\, a_2^2 + \frac{r_3}{r_2}\Delta_3 + \frac{r_4}{r_2} \Delta_4 \bigg)\\
 &= - \frac{r_2}{2r_0 r_3r_4}\, a_2^2 + \frac{r_1}{r_0}\Delta_1 + \frac{r_3}{r_0} \Delta_3 + \frac{r_4}{r_0}\Delta_4.
\end{align*}

 As above, the Hodge Index Theorem and the Bogomolov inequality now imply that $a_2$, $c_1(E_3)$, and $c_1(E_4)$ are bounded by some function  that depends only on $c_1(E)$, $c_2(E)$, $\phi_0$, and $K$.
 
Since torsion-free sheaves of rank one are semistable with respect to any polarisation, the process stops after at most $r-1$
  steps. It produces a filtration of $E$ with the property that the associated graduation has
   $\phi_0$-semistable torsion-free terms whose slopes with respect to $\phi_0$ are bounded by some constant $C = C(c_1(E), c_2(E), \phi_0, K)$ that depends only on $c_1(E)$, $c_2(E)$, $\phi_0$, and $K$.

 Finally, the inclusion $E_{\mathrm{max}}\subset E$ gives a nontrivial morphism from $E_{\mathrm{max}}$ to some term of this graduation showing that $\mu_0(E_{\mathrm{max}})\le  C$.
  \end{proof} 
  
As already noted above, Lemma~\ref{lem:boundedslope} implies Proposition~\ref{boundedness} by \cite[Prop.~3.3.7]{HL}. This concludes the proof of Proposition~\ref{boundedness}.
\end{proof}

\subsection{A chamber structure on the set of $(n-1)^{\mathrm{st}}$ powers of ample classes}\label{subsect:chamberstructure}
In the present section we will construct a chamber structure on $P(X)$ that reflects the change of the induced semistability condition, and we will investigate the basic properties of this decomposition.

\subsubsection{Constructing the chamber structure}
We first note the following fundamental relation between $\mathrm{Amp}(X)$ and $P(X)$.

\begin{Prop}[Injectivity of power maps]\label{P}
The set $P(X)$ is open in $N_1$, and the map $p_{n-1}\colon~\alpha\mapsto \alpha^{n-1}$ is a homeomorphism from $\mathrm{Amp}(X)$ to $P(X)$.
\end{Prop}
\begin{proof}
We put norms $\Vert \ \Vert_k$ on the real vector spaces $H^{k,k}_{\R}(X)$. For  $1\leq k\leq n$ the continuity of the maps $p_k\colon H^{1,1}_{\R}(X)\to H^{k,k}_{\R}(X)$, $\alpha\mapsto \alpha^k$ implies the existence of constants $C_k$ such that $\Vert  \alpha^k\Vert_k\leq C_k\Vert \alpha\Vert_1^k$ holds for all $\alpha$. Furthermore, we infer that the total derivative of $p_{n-1}$ at a point $\alpha$ is the map $\beta\mapsto (n-1)\alpha^{n-2}\cdot \beta$. Thus, the restriction of $p_{n-1}$ to the ample cone is a local isomorphism by the Hard Lefschetz Theorem. Consequently, the image $P(X)$ of $p_{n-1}$ is open in $N_1$. 

Let $\alpha,\beta$ be two real ample classes such that $\alpha^{n-1}=\beta^{n-1}$ in $N_1$. Multiplication by $\alpha$ from the left and by $\beta$ from the right gives
\begin{equation}\label{eq:mult}
\alpha^n=\alpha\beta^{n-1} \quad \text{ as well as } \quad \alpha^{n-1}\beta=\beta^{n},
\end{equation}
and hence
\begin{equation}\label{eq:npower}
 \alpha^n\beta^n=(\alpha^{n-1}\beta)(\alpha\beta^{n-1}).
\end{equation}
On the other hand, the Khovanskii-Teissier inequalities \cite[Ex.~1.6.4]{Lazarsfeld} give
\begin{equation}\label{eq:Teissier}
 (\alpha^{n-j}\beta^j)(\alpha^{n-j-2}\beta^{j+2})\le (\alpha^{n-j-1}\beta^{j+1})^2 \quad \text{for } 0\le j\le n-2.
\end{equation}
Multiplying all of these inequalities, we obtain the inequality
\begin{equation}\label{uneq:npower}
 \alpha^n\beta^n\le (\alpha^{n-1}\beta)(\alpha\beta^{n-1}).
\end{equation}
Note that in our setup \eqref{eq:npower} says that equality is attained in \eqref{uneq:npower}. Thus, equality must hold in each of the Khovanskii-Teissier inequalities \eqref{eq:Teissier} above. Together with the equalities \eqref{eq:mult}, this immediately implies that all the mixed intersection products $\alpha^{n-j}\beta^j$, $0\le j\le n$, are equal. It follows that $$(\alpha-\beta)\alpha^{n-1}=0;$$ i.e., $\alpha-\beta$ is primitive with respect to the polarisation $\alpha$. 
By the Hodge Index Theorem, the quadratic form $q(\gamma)\definiere \gamma^2\alpha^{n-2}$ is definite on the primitive part of $H^{1,1}_{\R}(X)$. Therefore, setting $\gamma=\alpha-\beta$ and invoking again the equality of mixed intersection products, we conclude that $\alpha=\beta$.
\end{proof}
\begin{rem}
 A differential-geometric argument proving an analogous result for compact K\"ahler manifolds was given in \cite{FuXiao}.
\end{rem}

As in the $2$-dimensional case, we obtain a locally finite linear rational chamber decomposition, this time however not on the ample cone, but on $P(X)$.
 \begin{thm}[Chamber structure on $P(X)$]\label{thm:chambers}
 For any set of topological invariants $(r, c_1, ..., c_n)$ of torsion-free sheaves on $X$ and for any compact subset $K \subset P(X)$, there exist finitely many linear rational walls defining a chamber structure with the following property: if two elements $\alpha$ and $\beta$ in $K$ belong to the same chamber then for any torsion-free coherent sheaf $F$ with the given topological invariants, $F$ is $\alpha$-(semi-)stable if and only if $F$ is $\beta$-(semi-)stable.
\end{thm}
\begin{proof} It suffices to consider the following situation: let $\phi_0$ be any real ample class in $N^1$, let $\hat K$ be a convex compact neighbourhood of $\phi_0$ in $\mathrm{Amp}(X)$, and let $K$ be its image under $p_{n-1}$ in $P(X)$. Using the notation introduced in the proof of Lemma~\ref{lem:boundedslope}, it follows from the arguments given there that change of semistability within $K$ occurs only at hyperplanes of the form $a_1^\perp \cap K$, see \eqref{eq:aisinkernel}. We have also seen loc.~cit.~that there are only finitely many such hyperplanes, once the discrete invariants of the sheaves and the compact set $K$ are fixed.
\end{proof}

\subsubsection{Explaining the pathologies found by Schmitt and Qin}\label{sect:Qin_and_Schmitt}
By Proposition \ref{P} our chamber structure on $P(X)$ pulls back to a locally finite chamber structure on $\Amp(X)$. The corresponding walls 
 thus obtained in $\Amp(X)$ are given by equations that are homogeneous of degree $n-1$, so, except in the case when $\rho(X)\definiere \dim N^1(X)\leq 2$, these need not be linear. This explains the pathologies encountered in the approaches of Schmitt and Qin. 

More precisely, on the one hand Schmitt \cite{Sch00} considers segments connecting rational points in $\Amp(X)$ as well as points on these segments where the induced notion of slope-stability changes. These separating points are precisely the intersection points of his segments with our walls. This clarifies the appearance of non-rational points as for example in \cite[Ex.~1.1.5]{Sch00}. On the other hand, these intersection points are also contained in the linear walls considered by Qin in \cite{Qin93}. This in turn explains the pathologies of Qin's linear chamber structure on $\Amp(X)$, and in particular the fact that it cannot be locally finite in general.  

\subsubsection{Representing chambers by complete intersection curves}
The system of walls given by Theorem~\ref{thm:chambers} yields an obvious stratification of $P(X)$ into connected chambers. We show next that every such chamber, even if it is not of the maximal dimension $\rho(X)$, contains a class which is an intersection of integral ample divisor classes, cf.~the discussion in Section~\ref{subsect:motivation}. This is one of the main motivations for the construction and investigation of a moduli space for families of sheaves that are slope-semistable with respect such a complete intersection class, as carried out in Sections~\ref{section:semiampleness} to \ref{sect:geometry} of this paper.
 \begin{Prop}[Representing chambers by complete intersection curves]\label{prop:a_cic_in_every_chamber}
 Let  $X$ be a projective manifold of dimension $n>2$ and fix some chamber $\mathcal{C} \subset P(X)$ of stability polarisations in $P(X)$. Then, there exist some ample integral classes $A$ and $B$ such that the complete intersection class 
 $A^{n-2}B$ lies in $\mathcal{C}$.
\end{Prop}
\begin{proof} 
For any $H\in\mathrm{Amp}(X)$, the $\R$-linear map $L_H\in L(N^1,N_1)$ given by $L_H(D)\definiere DH^{n-2}$ is invertible by Hard Lefschetz.
As the map 
$\mathrm{Amp}(X)\to L(N_1,N^1)$, $H\mapsto L_H^{-1}$ is continuous, the same also holds for the map
$$e\colon \mathrm{Amp}(X)\times N_1\to N^1, \ \ (H, C)\mapsto L_H^{-1}(C).$$
Take $H\in\mathrm{Amp}(X)$ such that $H^{n-1}$ lies in the fixed chamber $\mathcal{C} \subset P(X)$.
We have 
$e(H,H^{n-1})=H$. Since by Theorem~\ref{thm:chambers} the chambers are cut out by rational walls, close to 
$H^{n-1}$ there exists a rational element $C$ in the same chamber as $H^{n-1}$. Furthermore, choose a rational ample class $A\in\mathrm{Amp}(X)$ close to $H$.
Then, $B\definiere e(A,C)$ is close to $H$, and hence in particular, $B$ is in $\mathrm{Amp}(X)$. By construction, we have
\begin{equation}\label{eq:equivalentcompleteintersection}
 C=BA^{n-2}.
\end{equation}
 Since $C$ and $A$ are rational, we infer that the intersection numbers $BA^{n-2}D= CD$
are rational whenever $D\in N^1$ is rational. But the elements $A^{n-2}D$ span $N_1(X)_{\Q}$ as $D$ runs through $N^1(X)_{\Q}$. Hence, $B$ is a rational element in $\mathrm{Amp}(X)$. Together with equation~\eqref{eq:equivalentcompleteintersection} and with the observation that taking positive real scalar multiples in $\mathrm{Amp}(X)$ or $P(X)$ does not change the induced notion of slope-(semi)stability, this implies the claim.
\end{proof}
\vspace{0.5mm}
\def\cprime{$'$} \def\polhk#1{\setbox0=\hbox{#1}{\ooalign{\hidewidth
  \lower1.5ex\hbox{`}\hidewidth\crcr\unhbox0}}}
  \def\polhk#1{\setbox0=\hbox{#1}{\ooalign{\hidewidth
  \lower1.5ex\hbox{`}\hidewidth\crcr\unhbox0}}}
\providecommand{\bysame}{\leavevmode\hbox to3em{\hrulefill}\thinspace}
\providecommand{\MR}{\relax\ifhmode\unskip\space\fi MR }
\providecommand{\MRhref}[2]{%
  \href{http://www.ams.org/mathscinet-getitem?mr=#1}{#2}
}
\providecommand{\href}[2]{#2}

\vspace{1cm}

\end{document}